\documentclass[11pt,a4paper,reqno]{amsart}
\usepackage{fullpage}
\usepackage[utf8]{inputenc}
\usepackage[T1]{fontenc}
\usepackage{lmodern}
\usepackage{./common}
\usepackage[percent]{overpic}
\usepackage[foot]{amsaddr}

\address{Université Paris-Dauphine, CEREMADE, Place du Maréchal Lattre
  de Tassigny, 75775 Paris Cedex 16, France.}
\email{levitt@ceremade.dauphine.fr}
\thanks{Support from the grant
  ANR-10-BLAN-0101 of the French Ministry of Research is gratefully
  acknowledged}

\begin{document}
\title{Best constants in Lieb-Thirring inequalities: a numerical
  investigation}

\maketitle
\begin{abstract}
  We investigate numerically the optimal constants in Lieb-Thirring
  inequalities by studying the associated maximization problem. We use
  a monotonic fixed-point algorithm and a finite element
  discretization to obtain trial potentials which provide lower bounds
  on the optimal constants. We examine the one-dimensional and radial
  cases in detail. Our numerical results provide new lower bounds,
  insight into the behavior of the maximizers and confirm some
  existing conjectures. Based on our numerical results, we formulate a
  complete conjecture about the best constants for all possible values
  of the parameters.
\end{abstract}

\section{Introduction}
In this paper we study the family of Lieb-Thirring
inequalities\cite{lieb1975bound,lieb1976ineq}, which state that for
any potential $V \in L^{\gamma+\frac d 2}(\R^{d},\R)$ and non-negative
exponent $\gamma$,
\begin{align}
  \label{LT}
  \tr((-\Delta + V)_{-}^{\gamma}) \leq L_{\gamma,d} \int V^{\gamma + \frac d 2}_{-},
\end{align}
where $V_{-} = \max(-V,0)$ is the negative part of $V$, and $(-\Delta
+ V)_{-}^{\gamma}$ is the $\gamma$-th power of the negative part of
the operator $(-\Delta + V)$, as defined by the functional
calculus. In other words, $\tr((-\Delta + V)_{-}^{\gamma}) = \sum_{i}
|\lambda_{i}|^{\gamma}$ is the $\gamma$-th moment of the negative
eigenvalues of $-\Delta + V$.

This inequality was originally used by Lieb and Thirring in the case
$\gamma = 1, d = 3$ as an important tool to prove the stability of
fermionic matter in \cite{lieb1975bound} (see
\cite{lieb1990stability,liebseiringer} for an overview). The
generalization \eqref{LT} to general $\gamma$ and $d$ has since then
attracted a great deal of attention (see for instance
\cite{laptev2000recent} for a review). Of particular interest are the
values of $\gamma$ and $d$ for which this inequality holds, and the
value of the optimal constants $L_{\gamma,d}$.

Despite the physical significance of the Lieb-Thirring inequality and
the amount of mathematical research on the subject, many questions
remain. This paper aims to investigate some of these numerically,
something that, to our knowledge, has not been done since the work of
Barnes in the appendix of the original paper by Lieb and Thirring
\cite{lieb1976ineq}, in 1976.

We now briefly review what is known about $L_{\gamma,d}$. It is easy
to see that the bound \eqref{LT} cannot hold for $\gamma < 1/2$ when
$d = 1$ by scaling, and $\gamma = 0$ when $d = 2$ (because any
arbitrarily small potential in two dimension creates at least one
bound state). The inequality was proved for $\gamma > 1/2$ in one
dimension and $\gamma > 0$ in other dimensions in the original paper
by Lieb and Thirring\cite{lieb1976ineq}. The proof in the case $\gamma
= 1$ was recently greatly simplified by
Rumin\cite{rumin2011balanced}. The case $\gamma = 0, d \geq 3$
requires completely different methods and was proven independently by
Cwikel, Lieb and
Rosenblum\cite{cwikel1977weak,liebclr,rozenblum1972}. The limit case
$\gamma = 1/2, d = 1$ remained unsolved for twenty years until it was
finally settled by Weidl\cite{weidl1996}, and refined with a sharp
constant by Hundertmark, Lieb and Thomas\cite{hundertmark1998sharp}.

The inequality \eqref{LT} can be interpreted as a comparison of the
energy of the quantum mechanical system given by its Hamiltonian
$-\Delta + V$, with its semiclassical approximation that only involves
integrals of $V$. The semiclassical regime is obtained by letting the
Planck constant $\hbar$ tend to zero in the Hamiltonian $-\hbar\Delta
+ V$. In this paper, we have scaled $\hbar$ away from inequality
\eqref{LT} for convenience, and the corresponding limit is a large (or
extended) $V$. The eigenfunctions become localized in the region of
the phase space defined by $p^{2} + V(x) \leq 0$ and explicit
computations are possible in the limit. More precisely, using the Weyl
asymptotics, one can prove
\begin{align*}
  \lim\limits_{\mu \to \infty}\frac{\tr((-\Delta + \mu
    V)_{-}^{\gamma})}{\int (\mu V)^{\gamma + \frac d 2}} =  L_{\gamma,d,\text{sc}},
\end{align*}
where
\begin{align*}
  L_{\gamma,d,\text{sc}} = 2^{-d} \pi^{-d/2}
  \frac{\Gamma(\gamma+1)}{\Gamma(\gamma+\frac d 2 + 1)}.
\end{align*}
Therefore, denoting by $L_{\gamma,d}$ the optimal constant in
\eqref{LT}, we have $L_{\gamma,d} \geq L_{\gamma,d,\text{sc}}$, and we
set
\begin{align*}
  R_{\gamma,d} = \frac{L_{\gamma,d}}{L_{\gamma,d,\text{sc}}} \geq
  1.
\end{align*}

The value of $R_{\gamma,d}$ describes by how much the quantum
mechanical energy can exceed its semiclassical counterpart. This paper
is an investigation of $R_{\gamma,d}$. It aims at providing insight as
to its behavior by trying to solve numerically the variational problem
\begin{align}
  \tag{$P_{\gamma,d}$}
  \label{variational}
  R_{\gamma,d} = \frac 1 {L_{\gamma,d,\text{sc}}}\sup_{\int V^{\gamma + \frac d 2} = 1} \tr((-\Delta +
  V)_{-}^{\gamma}),
\end{align}
where we impose the condition $\int V^{\gamma + \frac d 2} = 1$ to
eliminate the scaling invariance of the Lieb-Thirring inequality.

Several important facts about $R_{\gamma,d}$ are known. A scaling
argument by Aizenman and Lieb \cite{aizenman1978semi} shows that
$R_{\gamma,d}$ is decreasing with respect to $\gamma$. Laptev and
Weidl \cite{laptev2000gammathreehalves} showed that $R_{\gamma,d} = 1$
for $\gamma \geq 3/2$. Helffer and Robert \cite{helffer1990riesz}
proved that $R_{\gamma,d} > 1$ for $\gamma < 1$ by expanding
$\tr(-\Delta + V)_{-}^{\gamma}$ for the harmonic potential $V(x) = 1 -
\abs{x}^{2}$ in the semiclassical limit. From these results, we can
deduce that for each dimension $d$ there exists a critical
$\gamma_{c,d}$ such that
\begin{align*}
  \begin{cases}
        R_{\gamma,d}> 1 & \text{when } \gamma < \gamma_{c,d},\\
        R_{\gamma,d} = 1 & \text{when } \gamma \geq \gamma_{c,d},\\
  \end{cases}
\end{align*}
with
\begin{align*}
  1 \leq \gamma_{c,d} \leq \frac 3 2.
\end{align*}

A trial potential of \cite{lieb1976ineq} provides a lower bound on
$R_{\gamma,d}$ and therefore $\gamma_{c}$. In the appendix of this paper,
Barnes solved numerically the restricted problem
\begin{align}
  \label{variational_1}
  R_{\gamma,d,1} = \frac 1 {L_{\gamma,d,\text{sc}}}\sup_{\int V^{\gamma + \frac d 2} = 1} (\lambda_{1})_{-}^{\gamma},
\end{align}
where $\lambda_{1}$ is the lowest eigenvalue of $-\Delta + V$. This is
equivalent to restricting \eqref{variational} to the potentials $V$
such that $-\Delta + V$ only has one negative eigenvalue. The solution
$V_{\gamma,d,1}$ of this problem is negative, radial and only has
bound state, \ie $-\Delta + V_{\gamma,d,1}$ only has one negative
eigenvalue. The corresponding $R_{\gamma,d,1}$ is decreasing, and
intersects $1$ at a critical $\gamma_{c,d,1}$. In low dimensions,
$\gamma_{c,1,1} = \frac 3 2$, $\gamma_{c,2,1} \approx 1.165$,
$\gamma_{c,3,1} \approx 0.863$ (our numerical results agree with these
values). Therefore, $\gamma_{c,1} = \frac 3 2$ and $\gamma_{c,2} \geq
1.165$, but nothing can be said about $\gamma_{c,d}$ for $d \geq 3$. A
famous conjecture of Lieb and Thirring states that $R_{\gamma,d} = 1,
$ \ie $\gamma_{c,3} = 1$. This would imply improved bounds on the
energy of a system of fermions, and is of great importance to the
Thomas-Fermi theory\cite{lieb1981thomas}.




Better upper bounds on $R_{\gamma,d}$ have also been derived
recently. For instance, it is proven that $R_{\gamma,d} \leq \frac {2
  \pi}{\sqrt 3} \approx 3.63$ for $\gamma \geq \frac 1 2$,
$R_{\gamma,d} \leq \frac \pi {\sqrt 3} \approx 1.82$ for $\gamma \geq
1$ (see \cite{dolbeault}), by generalizing the inequality to
matrix-valued potentials.

Despite these advances on upper bounds, not much is known about lower
bounds for $R_{\gamma,d}$, except for the one-bound state potential
$V_{\gamma,d,1}$ of \cite{lieb1976ineq} and the asymptotic result in
the semiclassical limit of \cite{helffer1990riesz}. In this paper, we
attempt to bridge the gap between these two results by looking
numerically for maximizers of the variational problem
\eqref{variational}. To our knowledge, this is the first numerical
study of the Lieb-Thirring inequalities since the work of
Barnes\cite{lieb1976ineq} (and unpublished recent work by Arnold and
Dolbeault \cite{dolbeaultcomm} in 1D).

The method we use is a finite element discretization of a fixed point
algorithm. We describe this algorithm, and specialize it to the case
of radial potentials. Then we discretize it, and use it to obtain
numerical results. The critical points we obtain are the natural
generalization of the potential with one bound state obtained by
Barnes. They serve as a partial bridge between this potential and
asymptotic results, and yield new lower bounds on $R_{\gamma,d}$.
\section{The optimization algorithm}
\subsection{Optimization scheme}
Let us denote $E(V) = \tr((-\Delta + V)_{-}^{\gamma})$, so that
$$L_{\gamma,d} = \sup_{\int V^{\gamma+\frac d 2} = 1} E(V).$$

The maximization algorithm is based on the following well-known
property:
\begin{proposition}
  \label{prop}
  For any $V \in L^{\gamma+\frac d 2}(\R^{d},\R)$ and any $\gamma \geq 1$,
  \begin{align*}
    E(V)^{1/\gamma} = \max_{\substack{\tau \geq 0 \\ \norm{\tau}_{\gamma'} = 1}} - \tr((-\Delta + V)\tau),
  \end{align*}
  where $\norm{\tau}_{\gamma'} = \lp \tr \lp|\tau|^{\gamma'}\rp\rp^{1/\gamma'}$ is the Schatten norm of the
  operator $\tau$ with exponent $\gamma' = \frac {\gamma}{\gamma - 1}$,
  the Hölder conjugate of $\gamma$.
\end{proposition}
\begin{proof}
  We have
  \begin{align*}
    -\tr((-\Delta + V)\; \tau) &\leq \tr ((-\Delta + V)_{-} \;\tau)\\
    &\leq \norm{(-\Delta + V)_{-}}_{\gamma} \norm{\tau}_{\gamma'}\\
    &\leq \tr((-\Delta + V)_{-}^{\gamma})^{1/\gamma}\\
    &\leq E(V)^{1/\gamma},
  \end{align*}
  and the equality is achieved when
  \begin{align}
    \label{tau_V}
    \tau = K_{\tau} (-\Delta + V)_{-}^{\gamma-1},
  \end{align}
  where $K_{\tau}$ is a normalization constant, chosen to ensure that
  $\norm{\tau}_{\gamma'} = 1$.

\end{proof}

From the proposition, we see that when $\gamma \geq 1$,
\begin{align*}
  (L_{\gamma,d})^{1/\gamma} = \sup_{\substack{\int V^{\gamma+\frac d 2} = 1,\\ \tau \geq
    0, \,\norm{\tau}_{\gamma'} = 1}} -\tr ((-\Delta+V)\tau).
\end{align*}

What we gain from this formulation is that we can explicitely maximize
$-\tr ((-\Delta+V)\tau)$ with respect to $V$ and $\tau$
separately. Indeed, for a fixed $V$, the maximum with respect to
$\tau$ is given by \eqref{tau_V}, and for a fixed $\tau$, the maximum
with respect to $V$ is given by
\begin{align*}
  V(x) = -K_{V} \tau(x,x)^{\frac 1 {\gamma + \frac d 2 - 1}},
\end{align*}
where  $K_{V}$ is a normalization parameter chosen to ensure that $\int
V^{\gamma + \frac d 2} = 1$, and $\tau(x,y)$ is the integral kernel of
$\tau$.

This suggests the following maximization scheme: Given an approximate
maximum $V_{n}$, we set $\tau_{n} = K_{\tau} (-\Delta +
V_{n})_{-}^{\gamma-1}$ and $V_{n+1}(x) = -K_{V} \tau_{n}(x,x)^{\frac 1
{\gamma + \frac d 2 - 1}}$, and iterate. Explicitely:
\theoremstyle{definition}
\newtheorem{algorithm}{Algorithm}
\begin{algorithm}{Maximization algorithm}
\begin{enumerate}
\item Compute the negative eigenvalues $\lambda_{i}$ and corresponding
  eigenvectors $\psi_{i}$ of $-\Delta + V_{n}$
\item Set $\rho_{n} = \sum_{i} (-\lambda_{i})^{\gamma-1} \psi_{i}^{2}$
\item Compute $K_{n} = \norm{\rho_{n}^{\frac 1 {\gamma + \frac d 2 -
        1}}}_{\gamma+\frac d 2}^{-1} = \lp\displaystyle\int \rho_{n}^{\frac {\gamma + \frac
    d 2} {\gamma + \frac d 2 -
        1}}\rp^{- \frac 1 {\gamma+\frac d 2}}$
\item Set $V_{n+1} = - K_{n} \rho_{n}^{\frac 1 {\gamma + \frac d 2 - 1}}$
\end{enumerate}
\end{algorithm}

By construction, this algorithm ensures that
\begin{align*}
  E(V_{n+1})^{1/\gamma} &= - \tr((-\Delta+V_{n+1}) \tau_{n+1})\\
  &\geq -\tr((-\Delta+V_{n}) \tau_{n+1})\\
  &\geq -\tr((-\Delta+V_{n}) \tau_n)\\
  &= E(V_{n})^{1/\gamma}
\end{align*}

Therefore, the sequence $E(V_{n})$ is non-decreasing, and since it is
bounded from above, it converges. In general though, $V_{n}$ may not
converge. We give examples in Section \ref{numres} where, because of
the translation invariance of the functional, $V_{n}$ splits into two
bumps that separate from each each other. Even when this behavior is
forbidden, for instance by imposing a finite domain as we do in
numerical computations, a rigorous convergence analysis of the
algorithm is still an open problem.

Note that, even if Algorithm 1 was derived from Proposition
\ref{prop} for $\gamma \geq 1$, it remains a reasonable algorithm when
$\gamma < 1$. In this case, though, the monotonicity of $E(V_{n})$ is
not guaranteed, and indeed was found to be false in numerical
computations. Nevertheless, we did not notice any specific convergence
issue for $\gamma < 1$.

Algorithm 1 can also be seen as a fixed-point scheme for the critical
points of the maximization problem. Indeed, the Euler-Lagrange
equations for \eqref{variational} are
\begin{align}
  \label{scf}
  V(x) = -K_{V} \left[(-\Delta+V)_{-}^{\gamma-1}(x,x)\right]^{\frac 1 {\gamma + \frac d 2 - 1}},
\end{align}
This is a self-consistent set of equations similar to systems such as
the Hartree-Fock equations of quantum
chemistry\cite{lieb1977hartree}. Our algorithm is similar in spirit to
the Roothaan method\cite{roothaan1951}. In our case, at least for
$\gamma \geq 1$, the scheme monotonically increases the objective
function. Therefore, the oscillatory behavior often seen in the
Hartree-Fock model\cite{cances2000,levitt} cannot occur here. Even for
$\gamma < 1,$ we did not see any such oscillations numerically, and
always observed linear convergence (\ie $\norm{V_{n} - V_{\infty}}
\leq \nu^{n}$ for some $\nu$, $0 < \nu < 1$), or slow separation of
bumps, as will be discussed in Section \ref{numres}.

This algorithm is also used in the context of the Lieb-Thirring
inequalities by Arnold and Dolbeault in 1D\cite{dolbeaultcomm}.
\subsection{Radial algorithm}
\label{radialalg}
Most of our multidimensional computations were done in a radial
setting. To see why this is appropriate, consider the above iteration
for $d \geq 2$, when $V_{n}$ is radial. Then, the Laplacian splits
into $\Delta_{r} + \frac 1 {r^{d-1}} \Delta_{\theta}$, where $\Delta_{r}$
is the radial Laplacian, and $\Delta_{\theta}$ the Laplace-Beltrami
operator on $S^{d-1}$. The eigenvectors of $\Delta_{\theta}$ are
explicitly known to be the functions whose orthoradial part are the
spherical harmonics, labeled by the integers $\ell$ and $m$. Since
$V_{n}$ is radial, $-\Delta_{r} + V_{n}$ commutes with
$\Delta_{\theta}$ and can therefore be diagonalized in the same basis
(separation of variables). We write these eigenvectors of the form
$\psi_{i,\ell,m}(x) = \phi_{i,\ell}(\abs x) J_{\ell,m}(x/\abs x)$, where
$J_{\ell,m}$ is the spherical harmonics of degree $\ell$ and order $m$. The
radial parts $\phi_{i,\ell}$ satisfy the equation
\begin{align}
  \label{eigen_l}
  - \frac 1 {r^{d-1}} (r^{d-1} \phi_{i,\ell}')' + \frac{\ell(\ell+d-2)}{r^{2}}\phi_{i,\ell} +
  V \phi_{i,\ell} = \lambda_{i,\ell} \phi_{i,\ell},
\end{align}
and each $\lambda_{i,\ell}$ has multiplicity
\begin{align*}
  h(d,\ell) = \binom{d+\ell-1}{\ell} - \binom{d+\ell-3}{\ell-2}
\end{align*}
(see \cite{stein1971introduction}.)

Therefore, we can obtain all the negative eigenvectors and eigenvalues
by solving only \eqref{eigen_l}. Furthermore, since the lowest
eigenvalue decreases as $\ell$ increases, we can iterate over $\ell$
and stop whenever the lowest eigenvalue becomes positive. Next, we
compute
\begin{align*}
  \rho_{n}(x) &= \sum_{\ell} \sum_{i} \sum_{m} (-\lambda_{i,\ell})^{\gamma-1}
  \psi_{i,\ell,m}(x)^{2}\\
  &= \sum_{\ell} h(d,\ell) \sum_{i}  (-\lambda_{i,\ell})^{\gamma-1} \phi_{i,\ell}(r)^{2}
\end{align*}
where the sum only ranges over all negative eigenvalues. This is again
radial, and therefore so is $V_{n+1}$. To summarize:

\begin{algorithm}{Radial maximization algorithm}
\begin{enumerate}
\item Compute the negative eigenvectors $\phi_{i,\ell}$ and
  eigenvalues $\lambda_{i,\ell}$ of \eqref{eigen_l} by increasing
  $\ell$ until the lowest eigenvalue  $\lambda_{0,\ell}$ becomes positive
\item Set $\rho_{n} = \sum_{\ell} \sum_{i} h(d,\ell)
  (-\lambda_{i,\ell})^{\gamma-1} \phi_{i,\ell}^{2}$
\item Compute $K_{n} = \norm{\rho_{n}^{\frac 1 {\gamma + \frac d 2 -
        1}}}_{\gamma+\frac d 2}^{-1} = \lp\displaystyle\int \rho_{n}^{\frac {\gamma + \frac
    d 2} {\gamma + \frac d 2 -
        1}}\rp^{- \frac 1 {\gamma+\frac d 2}}$
\item Set $V_{n+1} = - K_{N} \rho_{n}^{\frac 1 {\gamma + \frac d 2 - 1}}$
\end{enumerate}  
\end{algorithm}

Note that this algorithm is not an approximation: the iterates
generated by this algorithm coincide with the ones from Algorithm 1
when the initial guess $V_{0}$ is radial.
\section{Discretization}
\subsection{Galerkin basis and weak formulation}
To discretize this algorithm, we use a Galerkin basis on which we
expand $V$ and the eigenvectors. This discretization is variational at
two levels. First, the choice of eigenvectors is restricted to a
finite-dimensional subspace. By the min-max principle, this leads to
overestimation of the eigenvalues and therefore underestimation of
$E(V)$ for a given $V$. Second, the choice of potentials $V$ is also
restricted, decreasing the value of $\max_{\int V^{\gamma+\frac d 2} =
  1} E(V)$.

Our choice of discretization, rather than for instance finite
differences, has the advantage that numerical examples can provide
lower bounds on the best constants $R_{\gamma,d}$, with machine
precision as the only source of errors. A disadvantage is that the
algorithm we use involves taking powers of functions. There is no
exact way to do that in a Galerkin basis, and we must use an
approximation, which causes a loss of accuracy in the fixed-point
algorithm.

For the non-radial 1D and 2D cases, we simply use standard finite
elements. The radial case is less standard, and we must derive a weak
formulation of \eqref{eigen_l}. There are several possibilities here. The
simplest is to use a change of variable $\varphi(r) = r^{\frac{d-1}2}
\phi(r)$ to transform the equation back to the more standard form
\begin{align*}
  - \varphi'' + \frac {(l-\frac{d-1}{2})(l+\frac{d-3}{2})}{r^{2}}
  \varphi + V(r)\varphi = \lambda \varphi.
\end{align*}
We obtain a weak formulation by multiplying by a test function $u$ and
integrating:
\begin{align}
  \label{eigen_l_weak_bad}
  \int\left[ \phi' u' + \lp\frac {(l-\frac{d-1}{2})(l+\frac{d-3}{2})}{r^{2}}
    + V\rp \phi u\right] = \lambda \int \phi u,
\end{align}

Expanding the function $\phi$ on a Galerkin basis $\phi(r) = \sum_{i}
x_{i} \chi_{i}(r)$, this problem
transforms into the generalized matrix eigenvalue problem
\begin{align}
  \label{eigenproblem}
  A x = \lambda M x,
\end{align}
where
\begin{align}
  \label{stiffness}
  A_{ij} &= \int \chi_{i}' \chi_{j}' +
  \lp\frac{(l-\frac{d-1}{2})(l+\frac{d-3}{2})}{r^{2}} + V\rp\chi_{i}
  \chi_{j},\\
  \label{mass}
  M_{ij} &= \int \chi_{i}
  \chi_{j},
\end{align}
are the stiffness and mass matrices.

When $V$ is expanded on the same basis, we can compute the matrix
elements, solve the eigenvalue problem \eqref{eigenproblem}, and
transform back to $\phi$. However, for $d = 2, l = 0$, $\varphi$
behaves like $\sqrt r$ at $0$, and the singularity of the derivative
prevents an accurate discretization.

Another possibility is to obtain a weak formulation of \eqref{eigen_l}
directly. We have to remember that since the $\phi$ functions are only
radial parts of a $d$-dimensional function, the $L^{2}$ inner product
between two functions $\phi_{1}$ and $\phi_{2}$ is $\int \phi_{1}
\phi_{2} r^{d-1}$. This has to be taken into account in the weak
formulation in order to obtain a self-adjoint equation. Multiplying
\eqref{eigen_l} by $r^{d-1} u$, where $u$ is a test function, and
integrating by parts, we obtain:
\begin{align}
  \label{eigen_l_weak}
  \int\left[ r^{d-1} \phi' u' + \lp\frac{l(l+d-2)}{r^{2}} + V\rp r^{d-1}\phi u\right] = \lambda \int r^{d-1}\phi u.
\end{align}

Then, as with \eqref{eigen_l_weak_bad}, we can transform this into a
matrix equation and solve it. The disadvantage is that the
integrations required for the assembly step are more involved, and
therefore we only use it for the case $d = 2, l = 0$ where the other
method fails.

\subsection{Finite elements}
We use a Finite Element basis of piecewise linear functions on a grid
of $[0,L]$. The grid we use is a nonuniform grid of $N$ points, with
more points around $0$, to accommodate for the singularity of
\eqref{eigen_l}. $L$ is chosen large enough so that all the
eigenvectors associated to negative eigenvalues of $-\Delta + V$ can
be represented accurately in the basis. But if $\psi$ is an
eigenvector with negative eigenvalue $\lambda$, then, $\psi$ decays as
$\exp(-\sqrt{-\lambda} r)$. This shows that the discretization is
problematic for eigenvalues close to zero, as we shall see in Section
\ref{numres}.

For our piecewise linear basis functions, it is easy to compute the
matrix elements (\ref{stiffness},\ref{mass}) when $V_{n}$ is expanded
on this same basis $\chi_{i}$. We then solve the eigenproblem
\eqref{eigenproblem}, and obtain the eigenvectors. To expand
$\rho^{\frac 1 {\gamma + \frac d 2 - 1}}$ on the basis, we simply
chose the expansion that is exact on the grid points. Then the
normalization can be performed exactly, and the iteration is carried
out.

We present convergence results in \figref{fig:conv_NL}. As expected
from piecewise linear basis functions, we obtain $O(1/N)$ convergence of the
eigenfunctions $\phi$ in $H^{1}$ norm, and $O(1/N^{2})$ convergence of
the eigenvalues $\lambda$ (and therefore of $E(V)$). For a given
stepsize, the discretization error is exponential with respect to $L$,
with a decay rate equal to the decay rate of the eigenfunctions, \ie
$\sqrt{-\lambda}$. Although we used a uniform grid in one dimension in
\figref{fig:conv_NL}, similar results were checked to hold for our
non-uniform grid in $d$ dimensions, using the weak formulation
\eqref{eigen_l_weak_bad}, except for $d = 2, l = 0$, where the
convergence was slower and the weak formulation \eqref{eigen_l_weak} had to
be used to get the same convergence rates.

\begin{figure}[H]
  \centering
  \resizebox{\textwidth}{!}{ \includegraphics{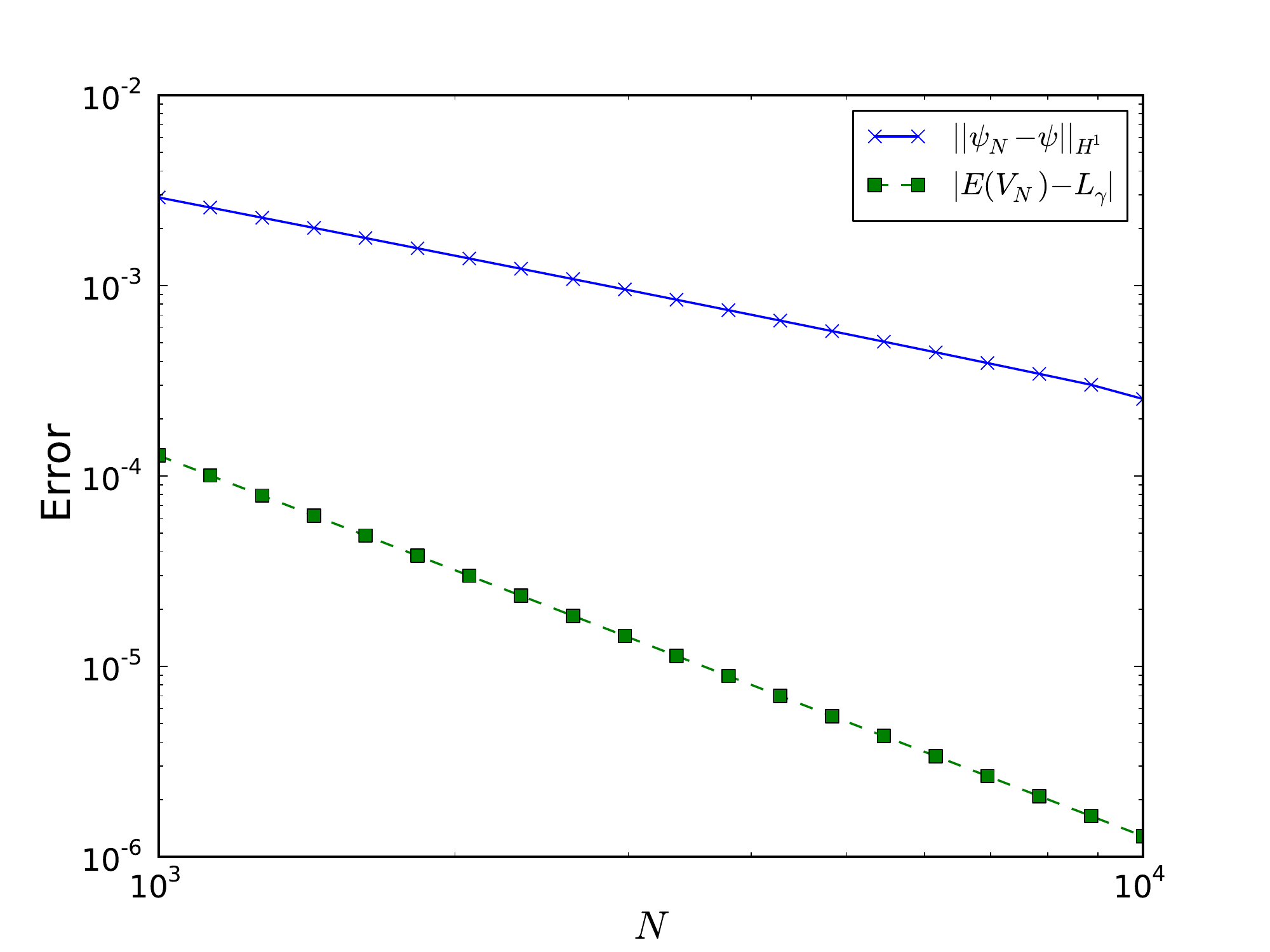} \includegraphics{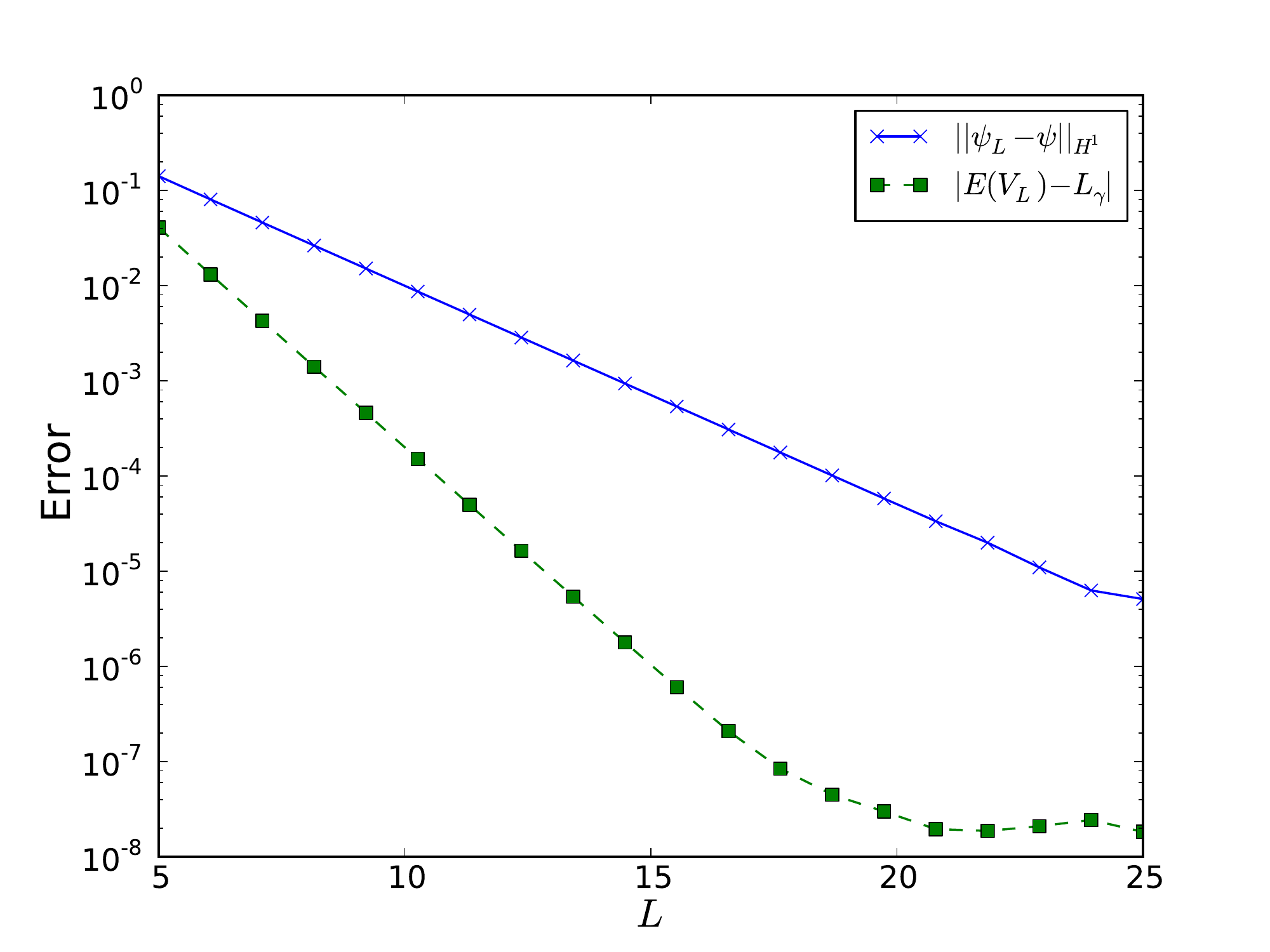}}
  \caption{Convergence analysis with respect to $N$ and $L$.  The
    figure on the left illustrates the convergence of the
    eigenfunctions and eigenvalues with respect to the number of grid
    points $N$, with a fixed $L = 40$, large enough to cause a
    negligible error. The slopes, found by linear regression, are
    $-1.043$ for the eigenvectors and $-1.999$ for the eigenvalues,
    close to their theoretical values of $-1$ and $-2$.  The figure on
    the right illustrates the convergence of the eigenfunctions and
    eigenvalues with respect to the domain size, with constant
    stepsize $h = 5 \times 10^{-4}$. The slope for the eigenvectors is
    $- 0.5285 \approx \sqrt{-\lambda} = -0.5283$. When the domain size
    gets large enough, the error due to $h$ dominates, and causes a
    plateau.}
  \label{fig:conv_NL}
\end{figure}

\subsection{Diagonalization}
The most computationally challenging step in our algorithm is the
problem of computing all  the negative eigenvalues and associated
eigenvectors of a generalized eigenvalue problem $A x = M x$, where
$A$ and $M$ are large sparse symmetric matrices. The spectrum of $A$
consists of $n$ negative eigenvalues, and a large number of positive
eigenvalues, which can be seen as perturbations of the spectrum of the
discrete Laplacian.

In order to compute the $n$ negative eigenvalues, where $n$ is
unknown, we compute the $k$ lowest eigenvalues, check if the $k$-th
one is positive, and repeat the process with a larger $k$ if not. The
computation of the $k$ first eigenvalues can be done by standard
packages such as ARPACK\cite{lehoucq1998arpack}. However, standard
algorithms such as Arnoldi iteration are not suited for the
computation of inner eigenvalues, and one has to solve for the largest
eigenvalues of the shifted and inverted problem $(A-\sigma)^{-1} x =
(\lambda-\sigma)^{-1} x$ instead to ensure reasonably fast
convergence. This requires an adequate shift (one that is close to the
bottom of the spectrum of $A$). Even with this shift-and-invert
strategy, the group of eigenvalues one needs to locate is not
well-separated from the rest of the spectrum, and becomes less and
less so as $L$ increases. This leads to slow convergence for large
$L$.

\subsection{Implementation}
We implemented the algorithm in Python, using the Numpy/Scipy
libraries\cite{scipy}, with the ARPACK eigenvalue
solver\cite{lehoucq1998arpack}.
\section{Numerical results}
\label{numres}

We used the algorithm described above to compute critical points of
the variational problem \eqref{variational}, that is, a solution to
the self-consistent equation \eqref{scf}. Our strategy was to find
a critical point by running the algorithm on a suitable initial
potential $V_{0}$ (for instance, a Gaussian of specified width). Once
convergence is achieved for a specific $\gamma$, we can run the
algorithm for $\gamma + \Delta \gamma$, using as initial potential the
one found for $\gamma$. We have been unable to analytically prove the
correctness of this method, for instance by checking the conditions of
the implicit function theorem. However, we found it reliable enough
for our purpose, as long as $\Delta \gamma$ was chosen small enough.

\subsection{The 1D case}
In one dimension we simply used a grid of size $[-L,L]$ with Dirichlet
boundary conditions. We reproduced the potential with one bound state
$V_{\gamma,1,1}$ of \cite{lieb1976ineq} by using for $V_{0}$ a
Gaussian of relatively small variance (see \figref{fig:conv_1BS}). In
this case, the algorithm converges linearly (\ie $\norm{V_{n+1} -
  V_{n}} \approx \nu^{n}$, for some convergence rate $\nu < 1$), and
very quickly (about twenty iterations to achieve machine precision).

\begin{figure}[H]
  \centering
    \resizebox{\textwidth}{!}{ \includegraphics{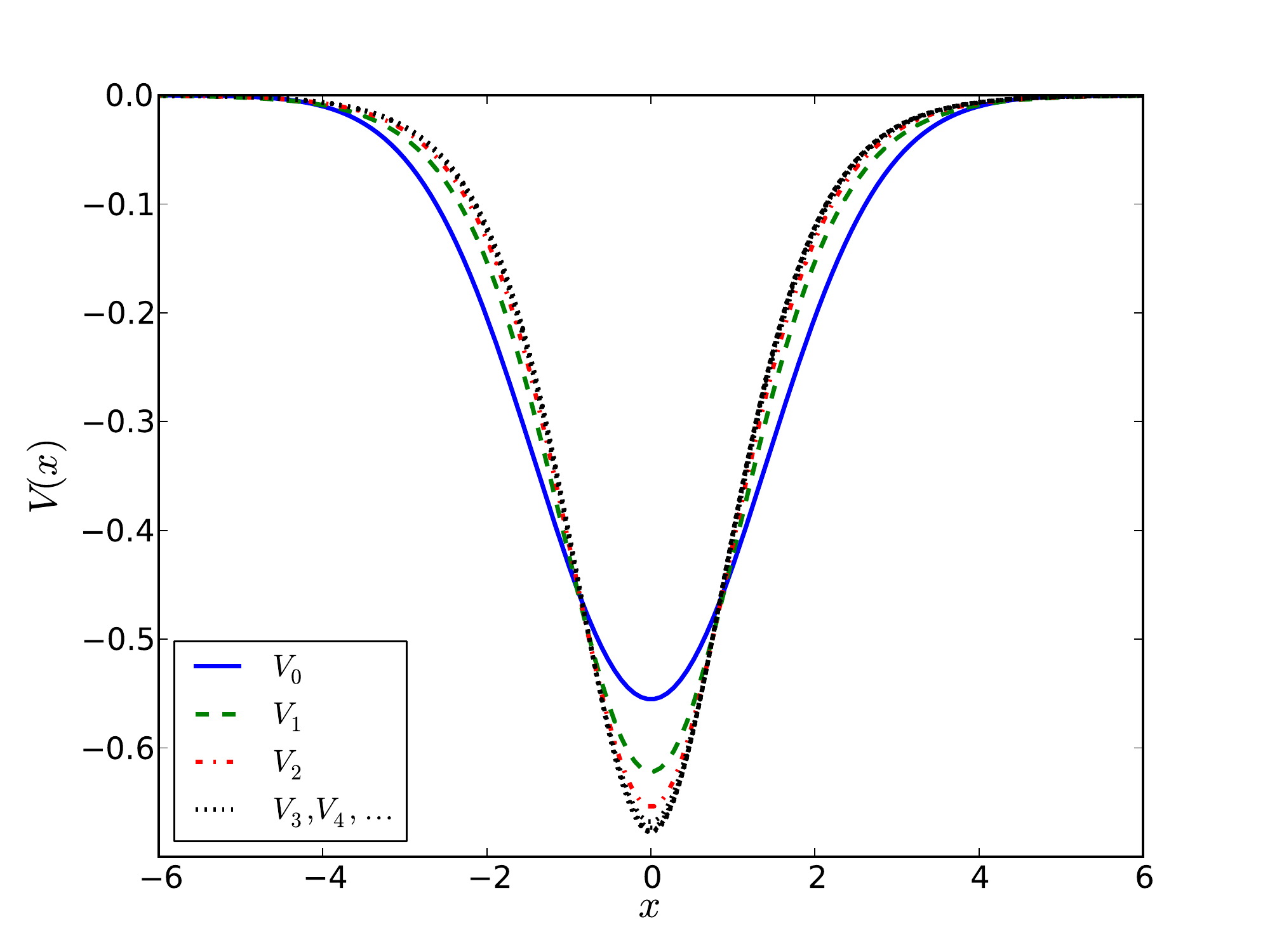}\includegraphics{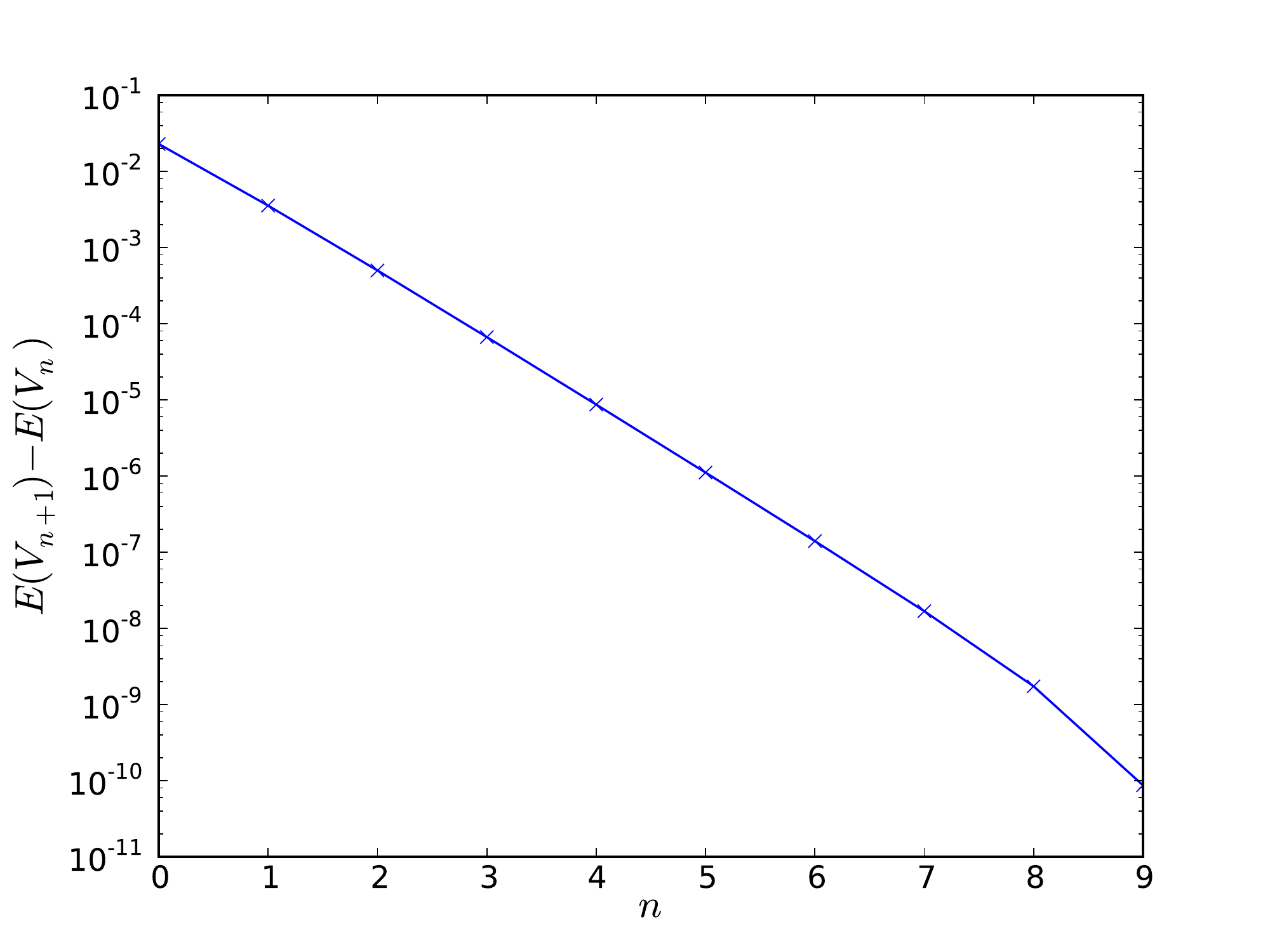}}
  \caption{Linear convergence from a Gaussian initial data towards
    $V_{\gamma,d,1}$. This plot is for $\gamma = 1.2$.}
  \label{fig:conv_1BS}
\end{figure}

Contrary to what we observe in higher dimensions, initializing the
algorithm with a Gaussian of larger variance did not make the
algorithm converge to other critical points. Instead, it leads to a
slow divergence where ``bumps'' separate, each bump corresponding to
the potential of $V_{\gamma,1,1}$ (see \figref{fig:div_1BS}).

\begin{figure}[H]
  \centering
    \resizebox{\textwidth}{!}{ \includegraphics{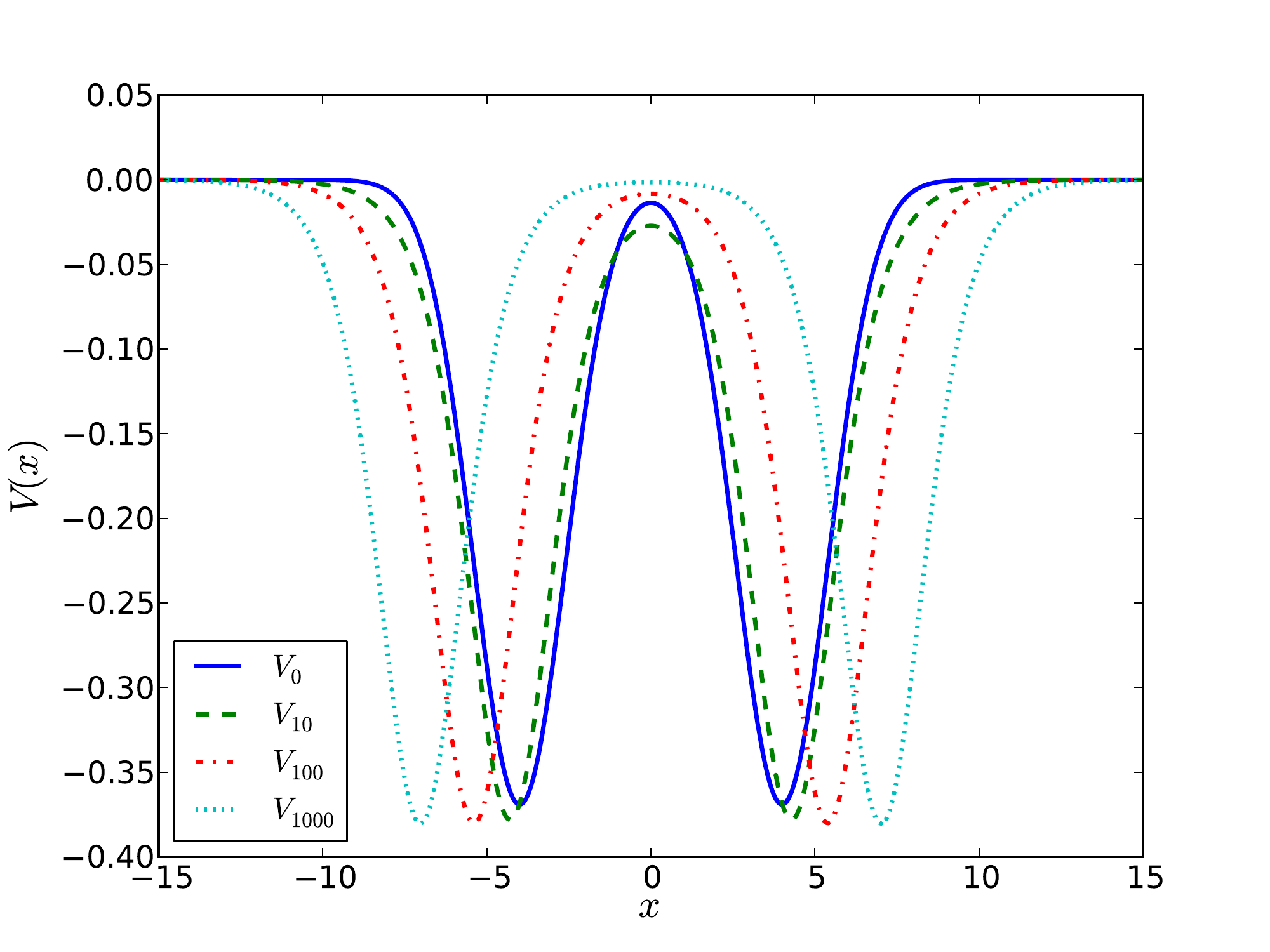}\includegraphics{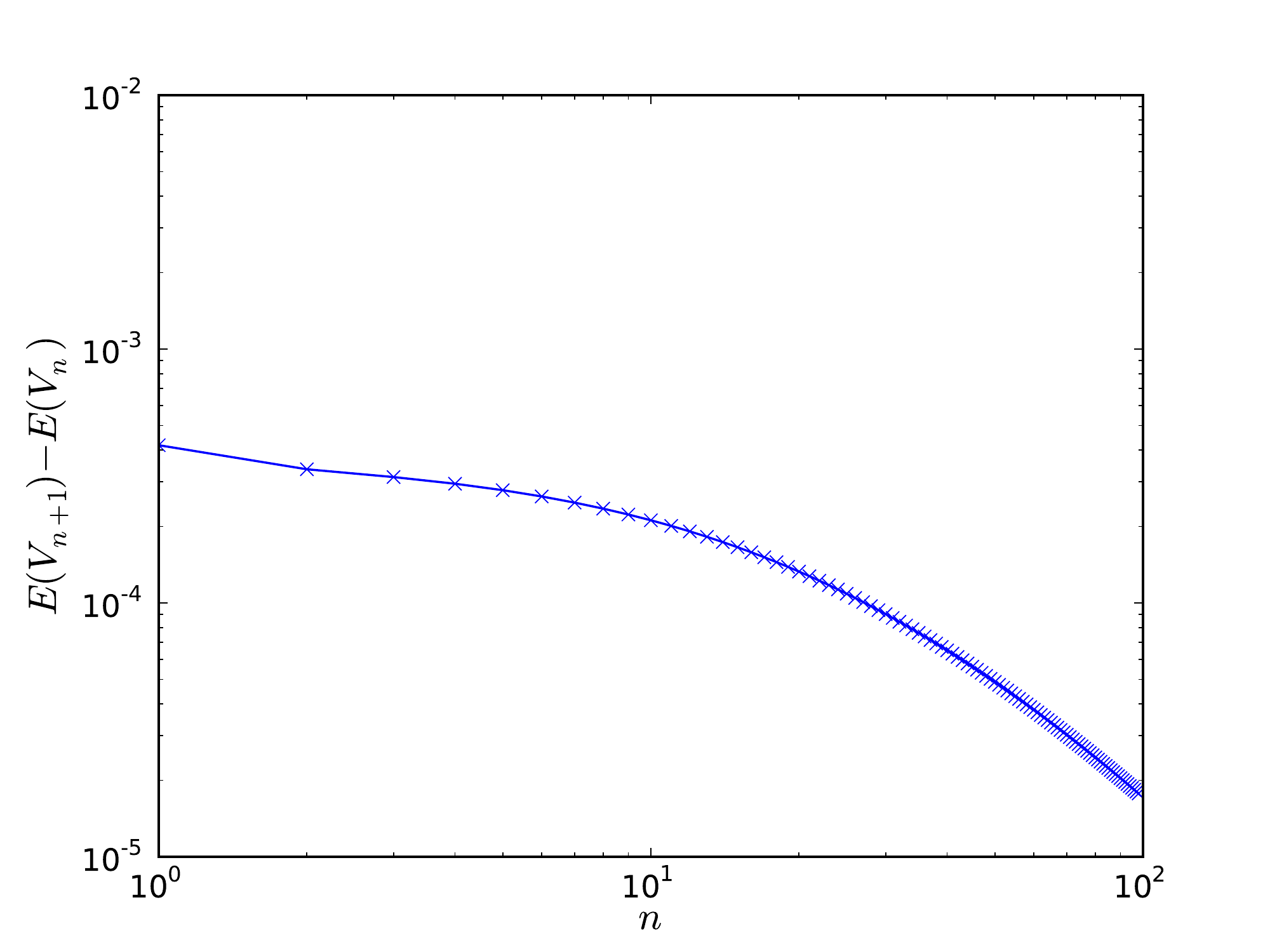}}
    \caption{Divergence from a sum of Gaussian bumps at $\gamma =
      1.2$. Initializing to a single Gaussian of large width yields
      similar results. The two bumps separate slowly, until the
      finiteness of $L$ forces convergence. Although not displayed
      here, the asymptotic logarithmic divergence
      \eqref{logarithmicsep} can be checked graphically, e.g. the
      spacing between $V_{100}$ and $V_{1000}$ is the same than
      between $V_{1000}$ and $V_{10000}$. The asymptotic slope of the
      log-log convergence plot is $-1$, which fits with the heuristic
      arguments given.}
  \label{fig:div_1BS}
\end{figure}

This effect occured regardless of the value of $\gamma$ in the range
$\gamma \in (1/2, 3/2)$. The divergence appears to be
logarithmic. More precisely, a numerical fit showed the asymptotic
relationship for the distance $L_{n}$ between the two bumps
\begin{align}
  \label{logarithmicsep}
  L_{n} \approx \frac 1 {2 \sqrt{-\lambda}} \log n,
\end{align}
where $\lambda$ is the unique negative eigenvalue of $- \Delta +
V_{\gamma,1,1}$.

This strikingly simple relationship can heuristically be understood by
the fact that, because the eigenfunctions $\phi_{1}$ and $\phi_{2}$
corresponding to the two bumps have exponential decay with decay
constant $\sqrt{-\lambda}$, their interaction is of order
$\exp(-\sqrt{-\lambda} L_{n})$. Due to cancellations, this interaction
leads to a correction of order $\exp(-2\sqrt{-\lambda} L_{n})$ in
$V_{n+1}$, and we have the approximate relationship $L_{n+1} \approx
L_{n} + K \exp(-2\sqrt{-\lambda} L_{n})$, which yields
\eqref{logarithmicsep}. A rigorous explanation of this is an
interesting question.

Based on these results, we conjecture that $V_{\gamma,d,1}$ is, up to
translation, the only maximizer of the functional. This would imply
that
\begin{align*}
R_{\gamma,1} = 2 \lp
\frac{\gamma-1/2}{\gamma+1/2}\rp^{\gamma-1/2}
\end{align*}
for $\gamma \leq \frac 3 2$, in agreement with the original conjecture of
Lieb and Thirring\cite{lieb1976ineq}.

\subsection{The 2D case}
Due to the high cost of accurately solving the maximization problem in
more than one dimension, we only obtained partial results in the
non-radial 2D case. The results indicate that the
algorithm either converges to radial critical points, or form slowly
separating bumps, as in one dimension. We have been unable to find any
example of non-radial critical points, although this does not mean
that they do not exist. An example of separation in bumps is provided
in \figref{fig:2D_sep}.

\begin{figure}[!h]
  \centering
  \resizebox{0.7\textwidth}{!}{ \includegraphics{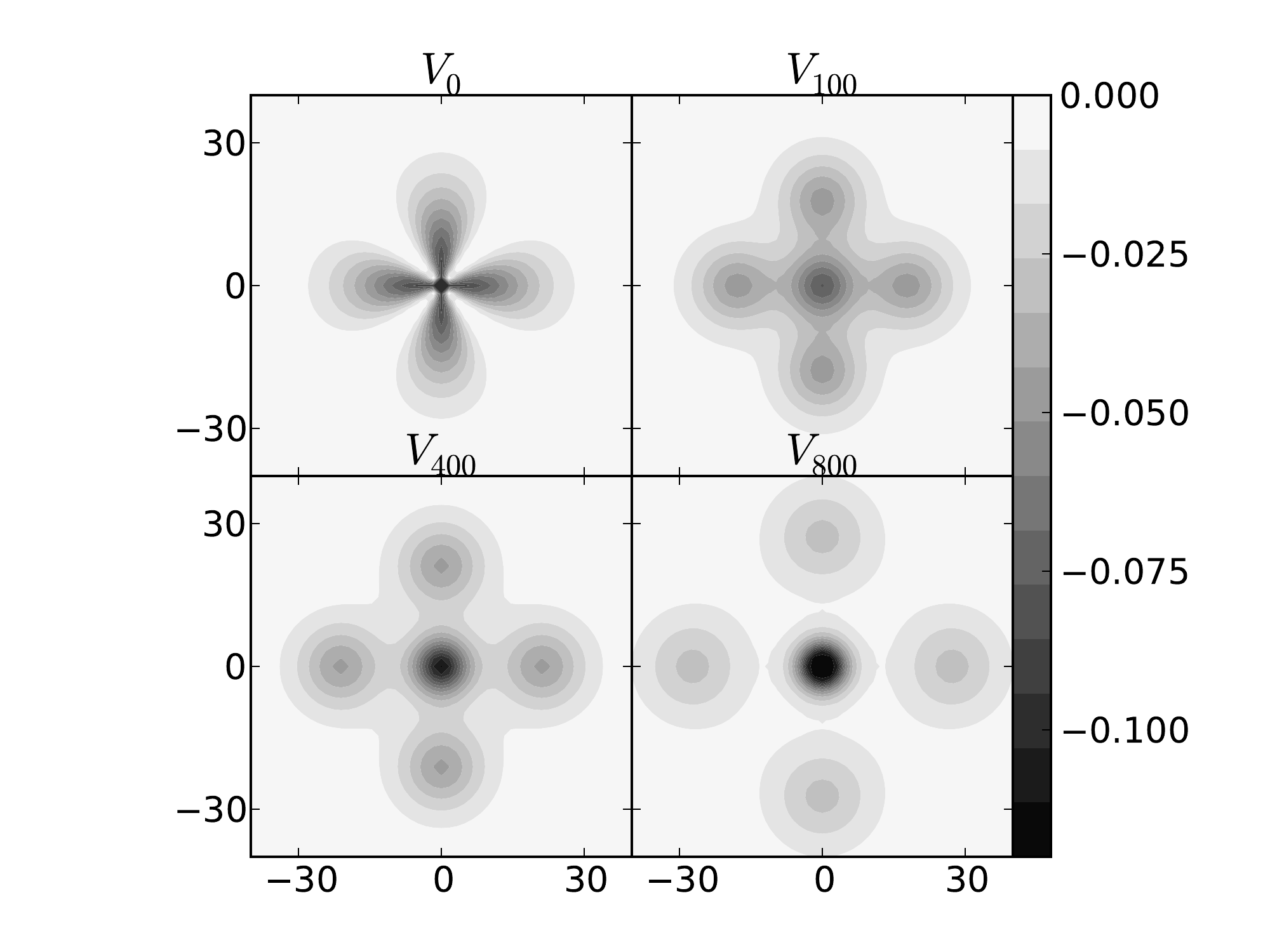}}
  \caption{Separation in bumps of 2D non-radial initial data at
    $\gamma = 1.2$, computed from Algorithm 1 with standard FEM. The
    initial data is taken to be a Gaussian in $r$ multiplied by an
    angular factor $(1+\cos(4\theta))$.}
  \label{fig:2D_sep}
\end{figure}



In the radial case, we followed branches of critical points of
\eqref{variational} using the continuation method on $\gamma$ we
described in Section \ref{numres}. We found this branch continuation
procedure robust. By varying the shape of the initial data, and in
particular its width, we were able to obtain different branches of
critical points. Generally speaking, as the width of the initial
potential increases, the number of negative eigenvalues of $-\Delta +
V$ increases.

We display the energy of these branches as a function of $\gamma$
in \figref{fig:2D}.

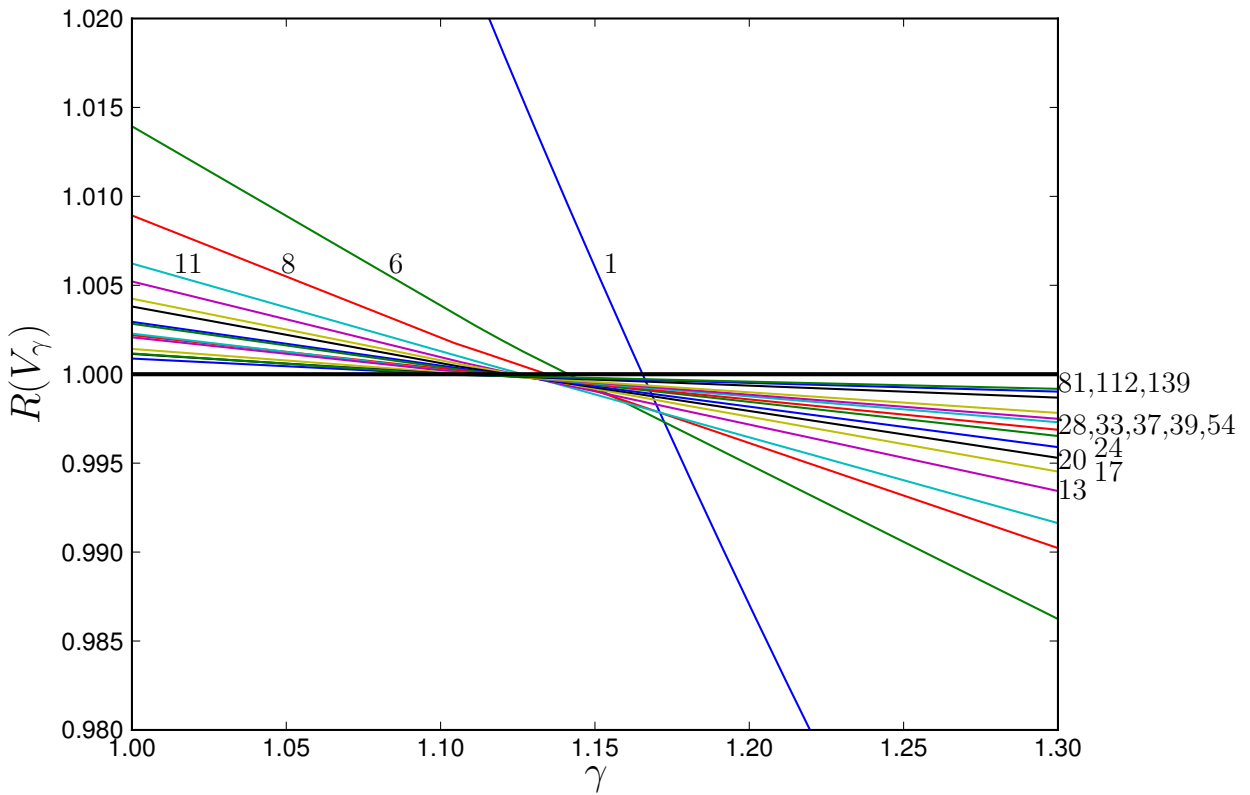
\begin{figure}[H]
  \centering
    \resizebox{\textwidth}{!}{
    \begin{overpic}
      {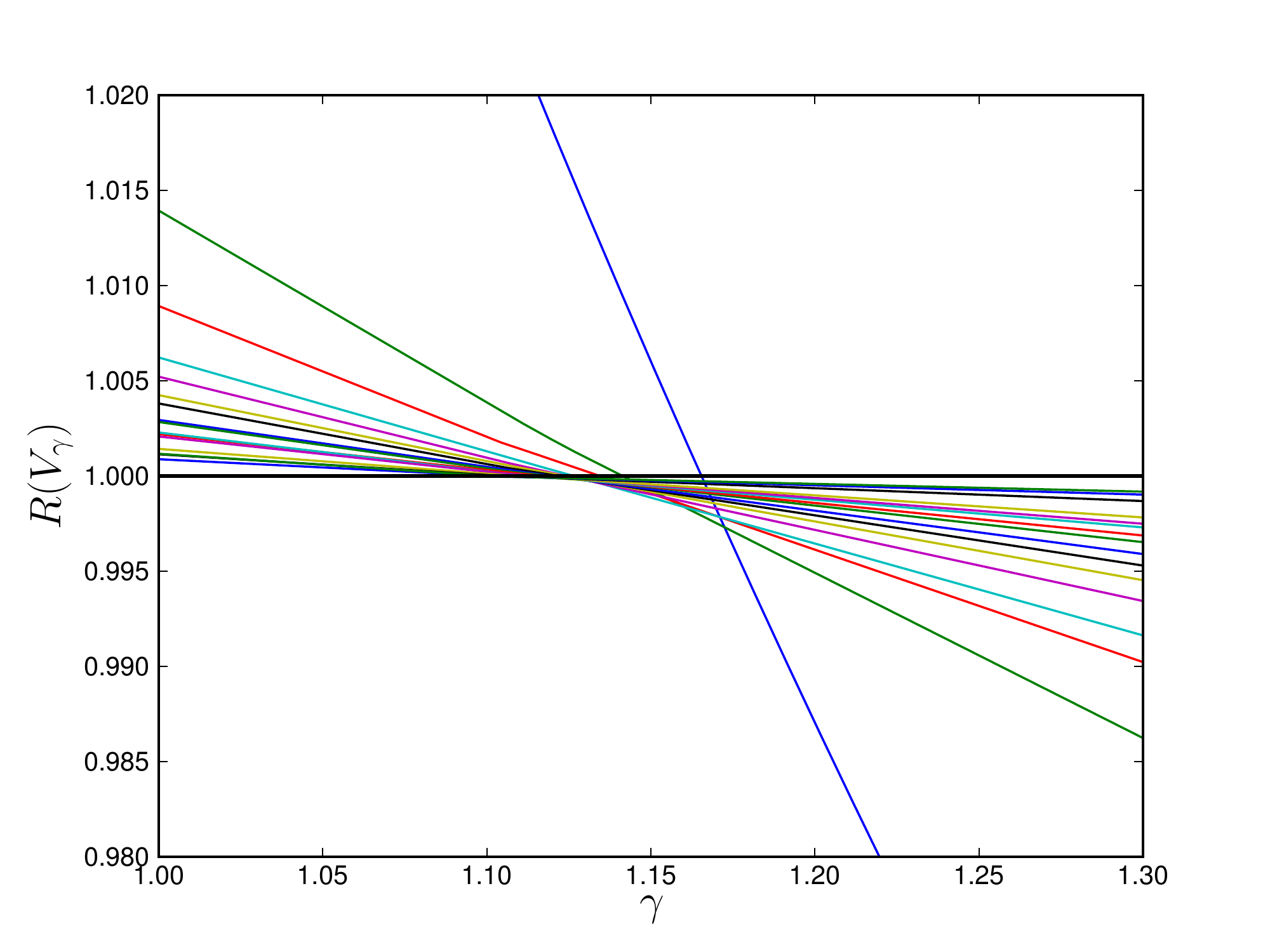}
      \Large
      \put(52,46){1}
      \put(34,46){6}
      \put(25,46){8}
      \put(16,46){11}
      \put(90,27){13}
      \put(93,28.5){17}
      \put(90,29.5){20}
      \put(93,30.5){24}
      \put(90,32.5){28,33,37,39,54}
      \put(90,36){81,112,139}
    \end{overpic}

  }
  \caption{$R(V)$ as a function of $\gamma$ for $d = 2$ using the
    branch continuation procedure in the radial setting (Algorithm
    2). The branches are labeled by the number of bound states they
    have. Based on these results, the maximizer seems to be
    $V_{\gamma,d,1}$, up to $\gamma \approx 1.16$. Above this, no
    branches were above the semiclassical regime $R = 1$.}
  \label{fig:2D}
\end{figure}

\begin{table}[H]
  \centering
  \begin{tabular}{|c|c|c|c|c|c|c|c|c|c|c|c|c|c|}
    \hline
    $k$            &1    &4    &6    &8    &11   &13   &17   &28   &54   &81   &139  \\
    \hline
    $\gamma_{c,2,k}$&1.165&1.150&1.141&1.135&1.126&1.124&1.124&1.119&1.110&1.105&1.100\\
    \hline
  \end{tabular}
  \caption{Values $\gamma_{c,2,k}$ of $\gamma$ at which some branches with $k$ bound
    states cross the threshold $R = 1$.}
  \label{tab:2Dcrossings}
\end{table}

First, we see that as the number of eigenvalues increases,
$R(V_{\gamma})$ tends to $1$, the semiclassical limit. For a given
$\gamma$, the Lieb-Thirring constant will be given by the supremum of
$R(V_\gamma)$ over such curves. From the branches depicted here, we see
that the supremum is always given by either $V_{\gamma,d,1}$ for
$\gamma < \gamma_{c}^{1} \approx 1.16$, and by the semiclassical limit
for $\gamma > \gamma_{c}^{1}$. All the other curves are below these
two (although not depicted in this zoomed plot, this holds true for
$0.5 < \gamma < 1.5$).

It is important to keep in mind that these branches only represent
critical points of the functional. They are generically local maxima
with respect to radial perturbations, but might not be stable with
respect to non-radial ones. For instance, \figref{fig:2D_sep_8BS}
depicts what happens when the radial potential with eight bound states
is perturbed non-radially by a Gaussian multiplicative noise, with
$\gamma = 1.1$. Because the potential with one bound state has a
higher energy, the potential splits into eight bumps. Performing the
same experiment $\gamma$ larger than about $1.17$, where the potential
with one bound state has a lower energy, when the potential with eight
bound states has a larger energy, we found that no splitting occurs,
which suggests that the potentials are stable beyond their crossing
with the one with one bound state.

\begin{figure}[H]
  \centering
  \resizebox{0.7\textwidth}{!}{ \includegraphics{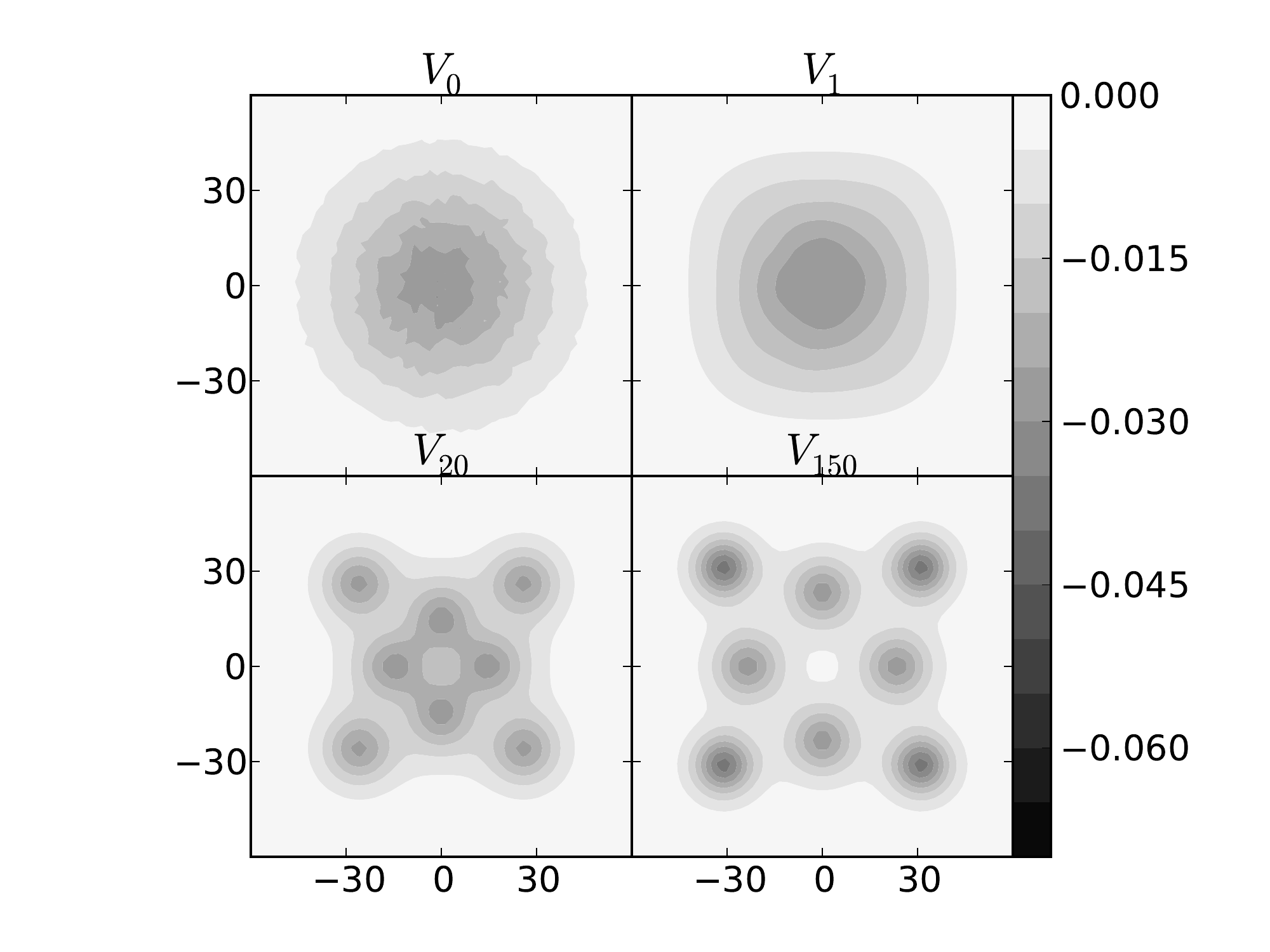}}
  \caption{Separation in bumps of the randomly perturbed radial
    potential with eight bound states, $\gamma = 1.1$.}
  \label{fig:2D_sep_8BS}
\end{figure}

Based on our numerical results, we conjecture that $\gamma_{c} =
\gamma_{c}^{1} \approx 1.16$. The only way this conjecture could be
false is if some other curve is above the one with one bound state. We
have not been able to find such a curve.

\subsection{The 3D case}
Due to the high cost of computation, we were unable to get meaningful
results in the non-radial setting, and only present our findings in
the radial case. Some of the radial potentials we found are presented
in \figref{fig:3D}, with numerical data about the crossings of the
curves in Table \ref{tab:3Dcrossings}.

\begin{figure}[H]
  \centering
    \resizebox{\textwidth}{!}{
    \begin{overpic}
      {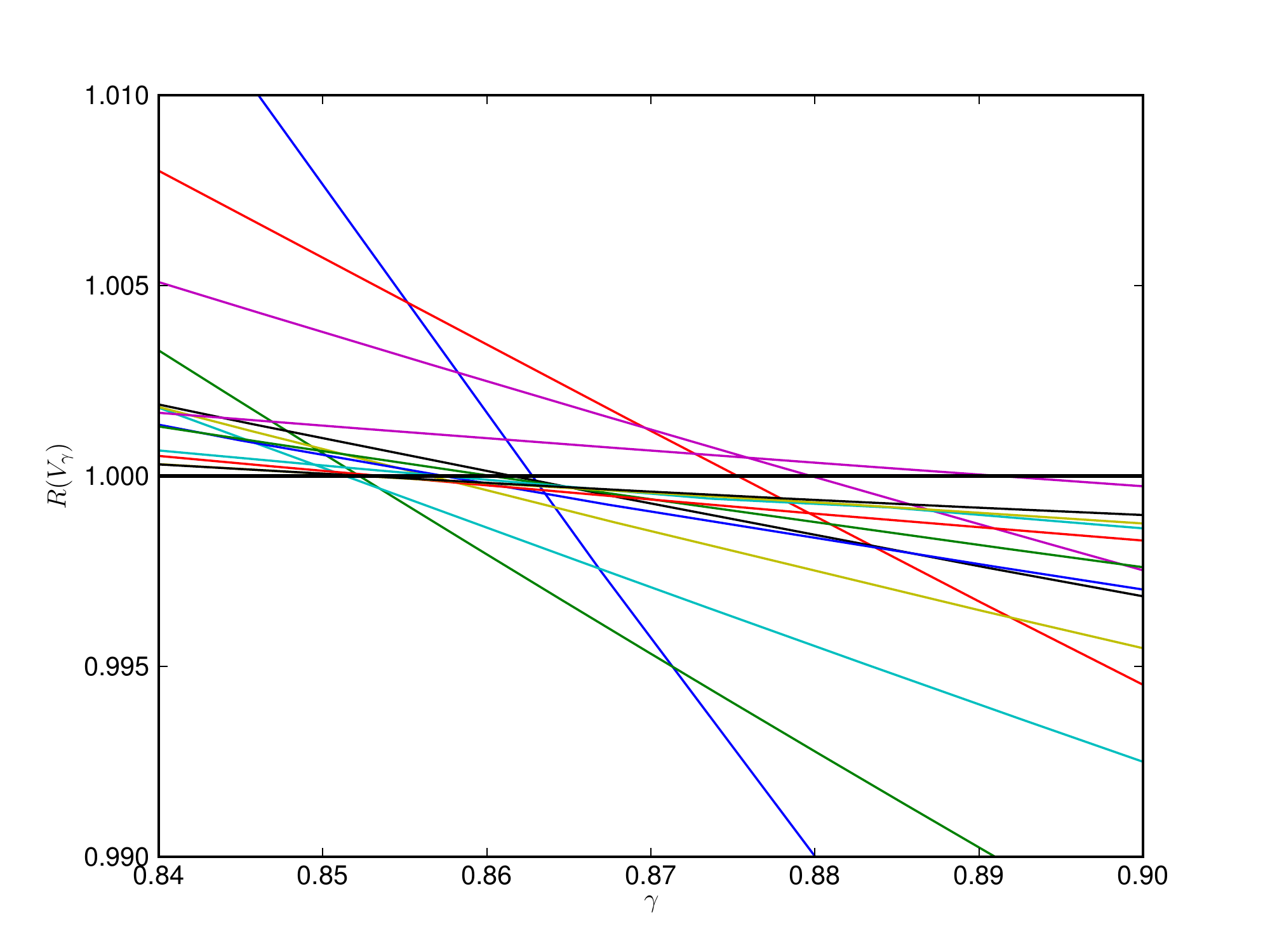}
      \Large
      \put(28,58){1}
      \put(72,12){4}
      \put(40,48){5}
      \put(76,20){10}
      \put(56,40){14}
      \put(91,24){21}
      \put(91,28){30,39}
      \put(91,30){55}
      \put(91,34){111,115,288,341}
      \put(64,39){140}
    \end{overpic}

  }
  \caption{$R(V)$ as a function of $\gamma$ for $d = 3$, with the same
    methodology as in \figref{fig:2D}. As $\gamma$ increases,
    different branches become maximizers, until $\gamma = 1$. The fact
    that only some branches are above the threshold $R = 1$ when the
    number of bound states increases results from the highly
    oscillatory behavior near $\gamma = 1$, as predicted in
    \cite{helffer1990riesz}.}
  \label{fig:3D}
\end{figure}
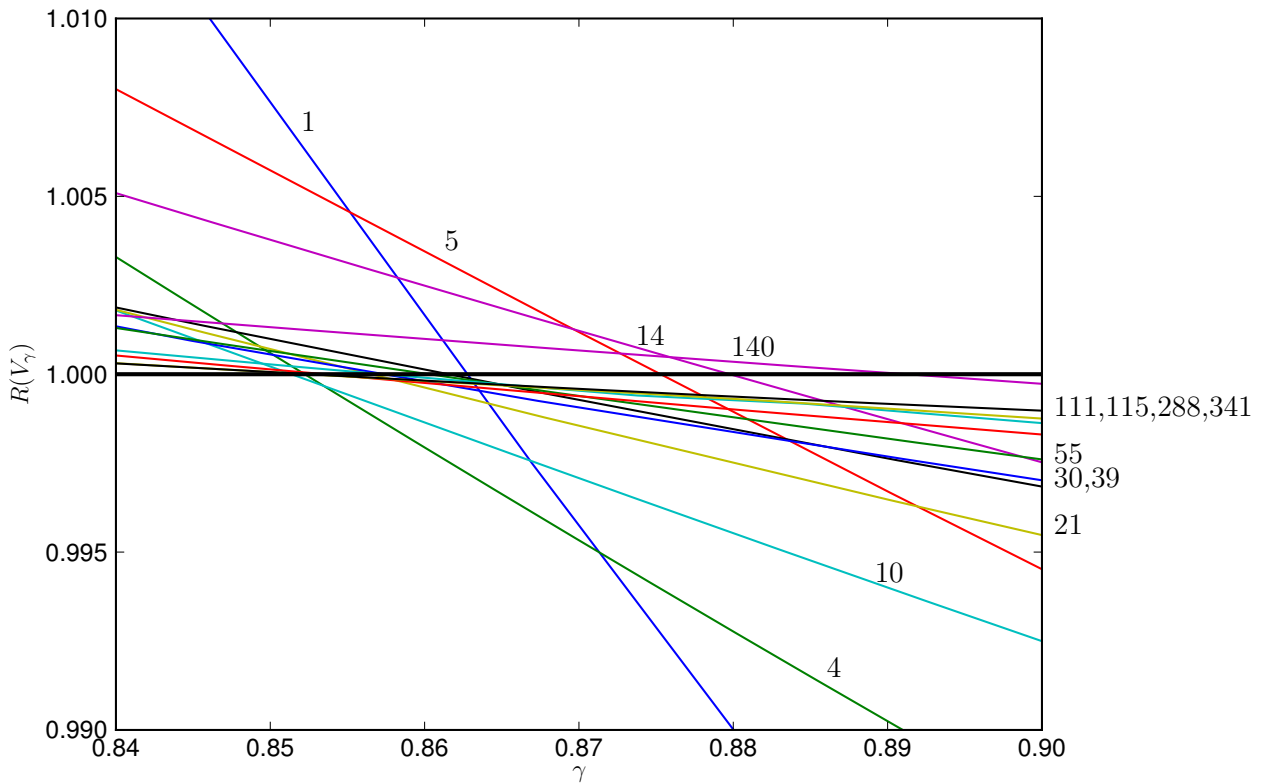

\begin{table}[H]
  \centering
  \begin{tabular}{|c|c|c|c|c|c|c|c|c|c|c|c|c|c|}
    \hline
    $k$            &1    &4    &5    &10    &14   &21   &30   &55   &111   &140   &341  \\
    \hline
    $\gamma_{c,3,k}$&0.863&0.852&0.875&0.851&0.880 &0.857&0.862&0.860&0.854&0.891&0.853\\
    \hline
  \end{tabular}
  \caption{Values $\gamma_{c,3,k}$ of $\gamma$ at which some branches with $k$ bound
    states cross the threshold $R = 1$.}
  \label{tab:3Dcrossings}
\end{table}

Contrary to the dimension 2, here some potentials with a higher number
of bound states have a higher $R$ than $V_{\gamma,d,1}$. The
corresponding curves in the $\gamma - R$ plane are flatter and
flatter, and intersect at a sequence of increasing
$\gamma_{c}^{k}$. This sequence seems to accumulate at $1$, in
accordance to Helffer and Robert's result\cite{helffer1990riesz},
which predicts the existence of similar potentials with a $R$ larger
than $1$ for every $\gamma < 1$ in the nearly semiclassical
regime. However, their study of an harmonic potential of varying width
showed a highly oscillatory behavior of $E(V)$ with respect to the
width of the potential. Although our numerical methods do not permit
us to investigate this regime (it would require a very large domain
size, with a correspondingly large number of points, and a very poor
conditioning of the eigenvalue problem), we expect the same behavior
to occur. This means that the computation of the maximizing branches
(here, those with 1, 5, 14 and 140 bound states) becomes harder and
harder.

The relative energies of the branches display a greater variety than
in the case $d = 2$, reflecting a more complicated energy
landscape. \figref{fig:maximizers} shows the profiles of some critical
potentials. Note that the potentials in dashed lines can be regarded
as ``anomalous'' versions of their similarly extended counterparts,
and have a lower energy. As can be expected, the potential branches
which are maximizers for some $\gamma$ have a profile that is
decreasing in $r$.

\begin{figure}[H]
  \centering
    \includegraphics[width=0.49\textwidth]{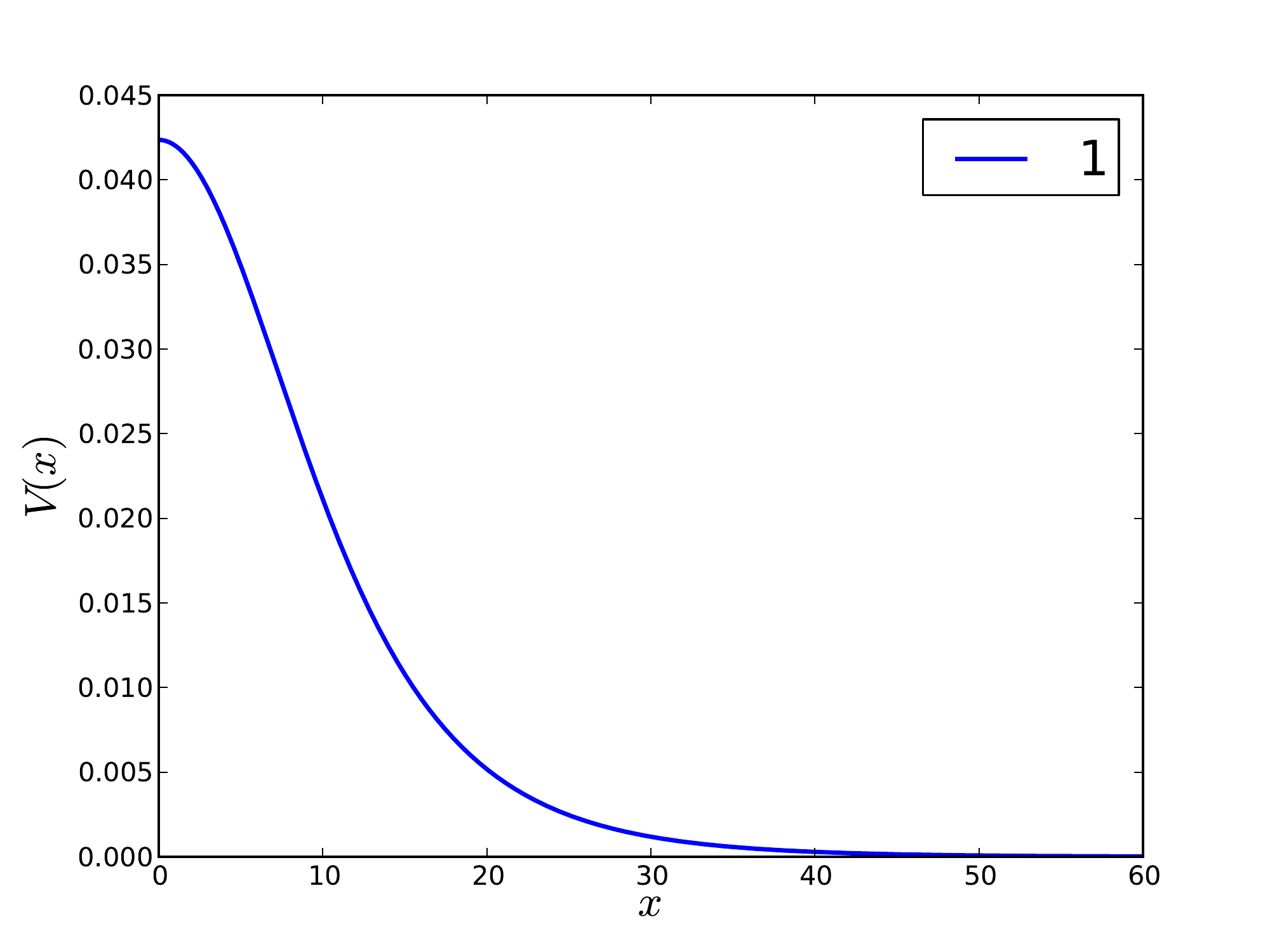}
    \includegraphics[width=0.49\textwidth]{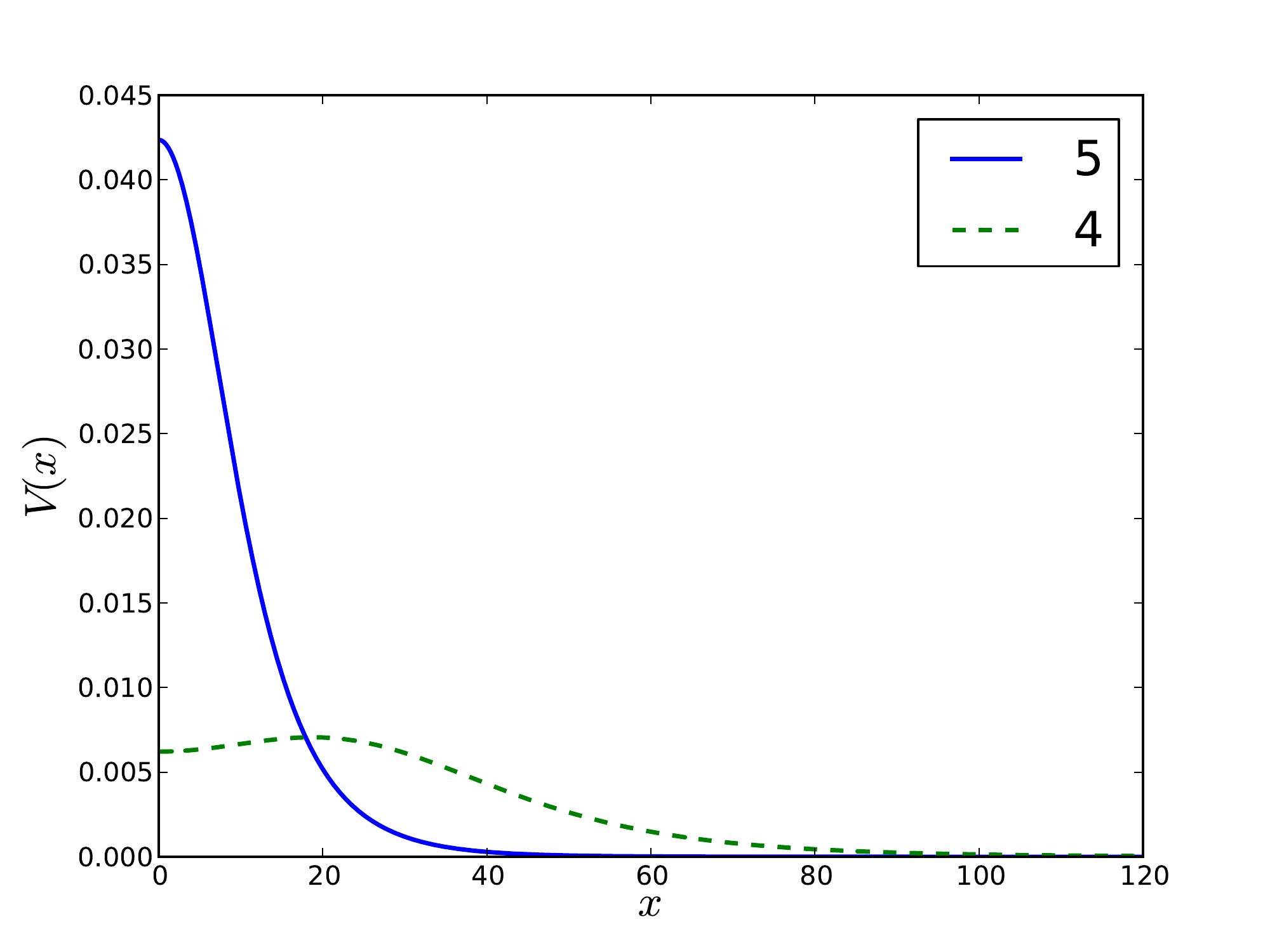}

    \includegraphics[width=0.49\textwidth]{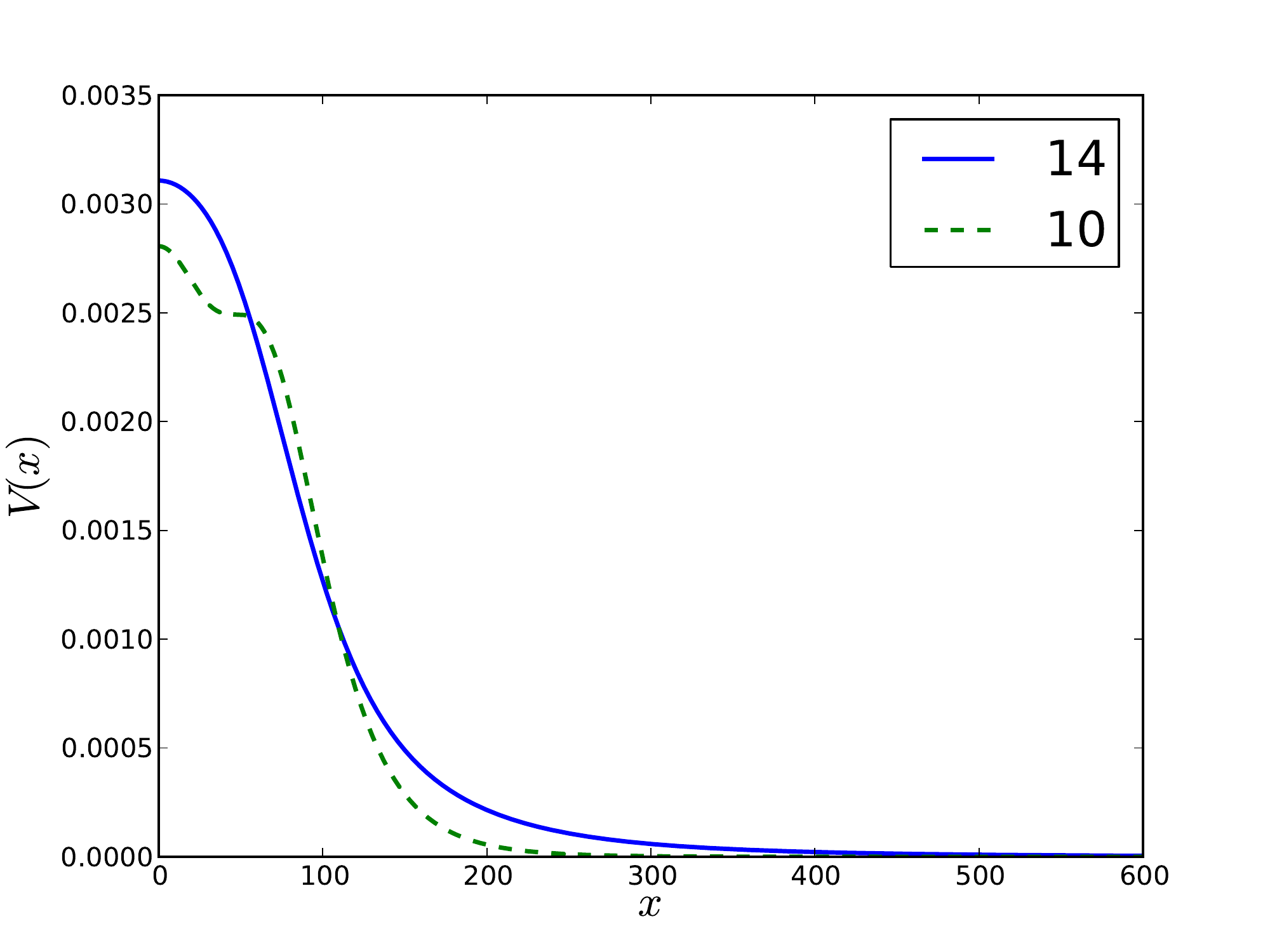}
    \includegraphics[width=0.49\textwidth]{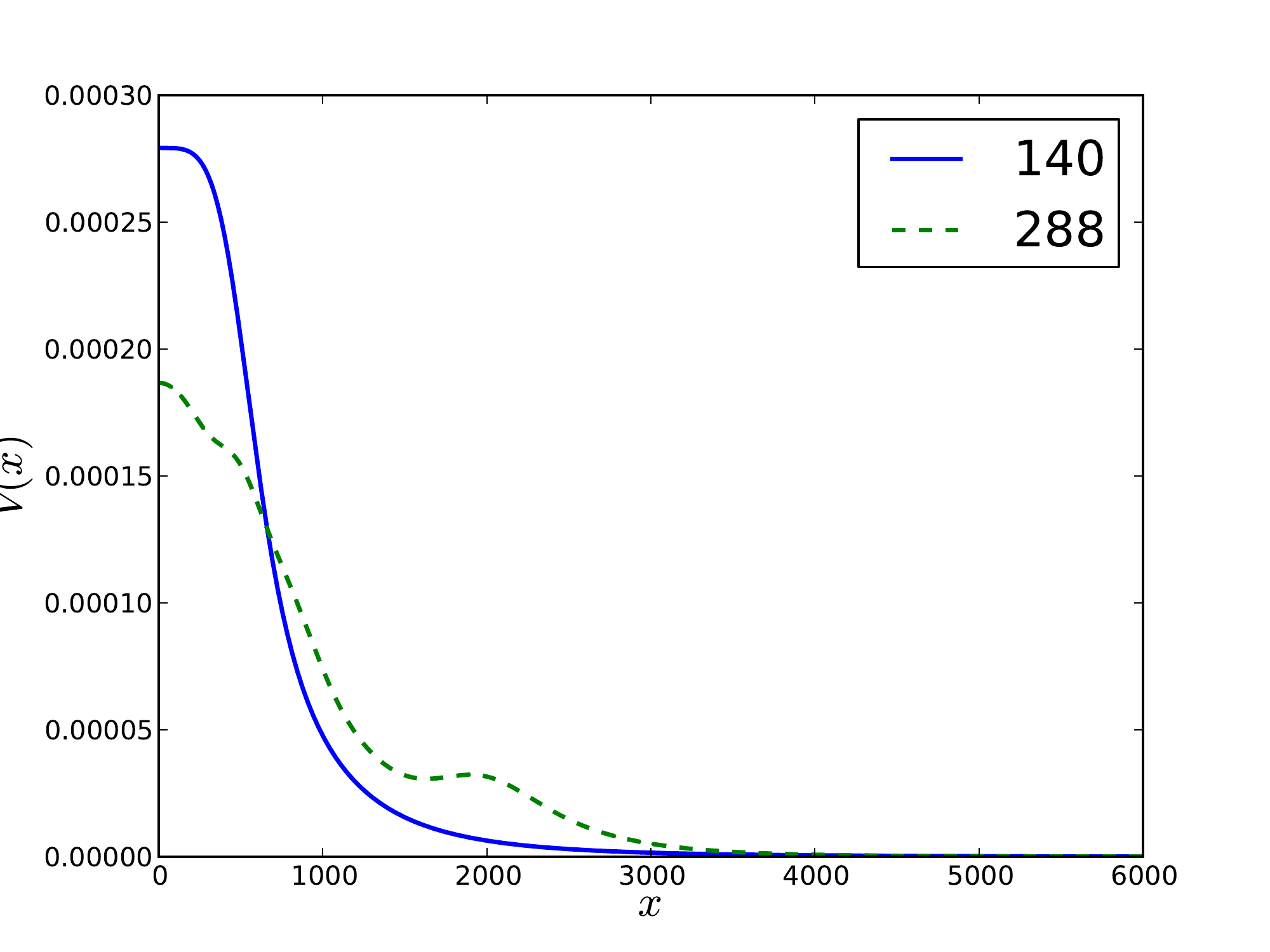}
    \caption{Shape of some critical points of $E(V)$, at $\gamma =
      1$. The qualitative shape of the curves does not change much
      when varying $\gamma$.}
  \label{fig:maximizers}
\end{figure}

Table \ref{tab:eigenvalues} displays the eigenvalue repartitions of a
particular potential, the one with 140 bound states, at $\gamma =
0.88$. Several patterns can be noted. First, for low values of $k$ and
$l$, the approximate relationship $\lambda_{k+1,l} = \lambda_{k,l+2}$
holds. This is because the associated eigenfunctions are localized
close to 0. In this region, the potential can be approximated by a
harmonic potential $V(0) + \frac 1 2 V''(0) x^{2}$, which leads to the
approximation $\lambda_{k,l} = V(0) + \sqrt{V''(0)} (2k+l + \frac 3
2)$, explaining the relationship $\lambda_{k+1,l} \approx
\lambda_{k,l+2}$. This is only valid for small $k$ and $l$, where the
harmonic approximation is valid. When the eigenvalues are close to
zero, another pattern emerges: the last negative eigenvalues have the
same value of $k+l$, leading to a triangular pattern in Table
\ref{tab:eigenvalues}. This seems to be true for the maximizing
branches (1, 5, 14, 140), but not for the others (which tend to have a
different ordering of the eigenvalues close to zero). We do not have
an explanation for this.

\begin{table}[H]
  \centering
  \begin{tabular}{|c|c|c|c|c|c|c|c|c|c|c|c|c|c|}
    \hline
    $\lambda_{k,l}$&$k = 1$    &$k = 2$    &$k = 3$    &$k = 4$    &$k =
    5$   &$k = 6$   &$k = 7$&$k = 8$\\
    \hline
    $l = 0$ &-85&-54&-30&-13&-4&-0.7&-0.02&+0.25\\
    \hline
    $l = 1$ &-69&-41&-20&-7.6&-1.8&-0.1&+0.02&\\
    \hline
    $l = 2$ &-54&-29&-12&-3.4&-0.3&+0.03&&\\
    \hline
    $l = 3$ &-39&-18&-5.8&-0.7&+0.04&&&\\
    \hline
    $l = 5$ &-26&-9&-1.3&+0.06&&&&\\
    \hline
    $l = 6$ &-14&-2.3&+0.08&&&&&\\
    \hline
    $l = 7$ &-3.8&+0.1&&&&&&\\
    \hline
  \end{tabular}
  \caption{Eigenvalue repartition for the branch with 140 bound
    states, with $\gamma = 0.88$, in units of $10^{-6}$ for readability. As in Section \ref{radialalg},
    $\lambda_{k,l}$ is the $k$-th eigenvalue of equation
    \eqref{eigen_l}. The positive eigenvalues are finite size effects
    and have no physical meaning, therefore we do not write them
    beyond the first one.}
  \label{tab:eigenvalues}
\end{table}

Every potential we were able to compute was below the $R = 1$
threshold for $\gamma > 1$. These results confirm the Lieb-Thirring
conjecture that $\gamma_{c,3} = 1$. However, the subtle behavior near
$\gamma = 1$ means that there might exist maximizing potentials for
$\gamma > 1$ we were unable to find.

\subsection{The case $d \geq 4$}
The results we obtained in the case $d \geq 4$ are similar to the case
$d = 3$, with $V_{\gamma,d,1}$ as the maximizer for small $\gamma$,
until it is outperformed by branches with a larger number of bound
states, which all fall under the semiclassical limit before $\gamma =
1$. However, the high cost of computation (the higher the dimension,
the more ill-conditioned the eigenvalue problem is) prevents us from
performing a systematic study.
\section{Conclusion}
In this paper, we used a maximization algorithm to investigate the
best constants of the Lieb-Thirring inequalities, an important open
problem of quantum mechanics. Discretizing the problem using finite
elements, we were able to numerically compute critical points of the
functional. Our main findings are the following:
\begin{itemize}
\item In one dimension, the only critical point we found is the
  potential with exactly one bound-state $V_{\gamma,1,1}$. For all
  initial data, the algorithm seems to either converge to this
  potential, or to split in separating bumps. This supports the
  conjecture that these potentials are the only maximizers of the
  functional for $\frac 1 2 \leq \gamma < \frac 3 2$.
\item In two dimensions, the only critical points we found were
  radial. For all initial data, the algorithm seems to either converge
  to a radial potential, or to split in diverging bumps. Among the
  radial potentials we were able to compute, $V_{\gamma,2,1}$ was the
  maximizer for $\gamma < 1.16$. After this point, the maximum
  corresponds to the semiclassical regime. The natural conjecture is
  that $V_{\gamma,d,1}$ is the maximizer for $\gamma < 1.16$, and
  therefore that $\gamma_{c,2} \approx 1.16$.
\item In three dimensions, branches of radial potentials with more
  than one bound state outperformed $V_{\gamma,d,1}$. This confirms
  the theoretical result from \cite{helffer1990riesz}, and provides
  new lower bounds (see \figref{fig:3D}). The
  natural conjecture is that there is a non-decreasing sequence of
  integers $n_{\gamma}$ with $n_{\gamma} \to \infty$ as $\gamma \to 1$
  such that the maximizer of the functional \eqref{variational} has
  $n_\gamma$ bound states.
\item In dimension $d \geq 4$, the results are similar to the
  three-dimensional case, and we conjecture the same behavior.
\end{itemize}

Beyond these results, the study hinted at a very rich and highly
nonlinear behavior of the maximizers of Lieb-Thirring
inequalities. This study is based on a simple numerical method (finite
element discretization). More involved computations could use a more
appropriate Galerkin basis, such as the Gaussian basis functions used
in quantum chemistry. This could allow for a more accurate computation
of potentials with a larger number of bound states, and a more
detailed exploration of the energy landscape.

On the theoretical side, the properties of the functional
\eqref{variational} remain unexplored. An open question is whether one
can prove the existence of maximizers. The behavior of the
maximization algorithm is also an interesting question, with the
separation in bumps particularly interesting. Finally, we believe that
our method could be adapted to generalizations such as models with
positive temperatures or positive density\cite{frank2011positive},
negative values of $\gamma$\cite{dolbeault2006lieb} or polyharmonic
operators $(-\Delta)^{l} + V$ (see \cite{netrusov1996lieb}).

\section*{Acknowledgments}
I would like to thank Mathieu Lewin and Éric Séré for their help and
advice.

\bibliographystyle{plain}
\bibliography{refs_lt.bib}

\begin{thebibliography}{10}

\bibitem{aizenman1978semi}
M.~Aizenman and E.H. Lieb.
\newblock {O}n semi-classical bounds for eigenvalues of {S}chr{\"o}dinger
  operators.
\newblock {\em Phys. Lett. A}, 66(6):427--429, 1978.

\bibitem{cances2000}
E.~Canc{\`e}s and C.~Le~Bris.
\newblock {O}n the convergence of {SCF} algorithms for the {H}artree-{F}ock
  equations.
\newblock {\em ESAIM Math. Model. Numer. Anal.}, 34(4):749--774, 2000.

\bibitem{cwikel1977weak}
M.~Cwikel.
\newblock {W}eak type estimates for singular values and the number of bound
  states of {S}chrodinger operators.
\newblock {\em Ann. of Math.}, pages 93--100, 1977.

\bibitem{dolbeaultcomm}
J.~Dolbeault.
\newblock personal communication, 2012.

\bibitem{dolbeault2006lieb}
J.~Dolbeault, P.~Felmer, M.~Loss, and E.~Paturel.
\newblock {L}ieb--{T}hirring type inequalities and {G}agliardo--{N}irenberg
  inequalities for systems.
\newblock {\em J. Funct. Anal.}, 238(1):193--220, 2006.

\bibitem{dolbeault}
J.~Dolbeault, A.~Laptev, and M.~Loss.
\newblock {L}ieb-{T}hirring inequalities with improved constants.
\newblock {\em J. Eur. Math. Soc. (JEMS)}, 10:1121--1126, 2008.

\bibitem{frank2011positive}
R.L. Frank, M.~Lewin, E.H. Lieb, and R.~Seiringer.
\newblock {A} positive density analogue of the {L}ieb-{T}hirring inequality.
\newblock {\em Duke Math. J., to appear}, 2012.

\bibitem{helffer1990riesz}
B.~Helffer and D.~Robert.
\newblock {R}iesz means of bounded states and semi-classical limit connected
  with a {L}ieb-{T}hirring conjecture {II}.
\newblock {\em Ann. Inst. Henri Poincaré (A)}, 53(2):139--147, 1990.

\bibitem{hundertmark1998sharp}
D.~Hundertmark, E.H. Lieb, and L.E. Thomas.
\newblock {A} sharp bound for an eigenvalue moment of the one-dimensional
  {S}chrodinger operator.
\newblock {\em Adv. Theor. Appl. Math.}, 2:719--732, 1998.

\bibitem{scipy}
E.~Jones, T.~Oliphant, P.~Peterson, et~al.
\newblock {SciPy}: {O}pen source scientific tools for {Python}, 2001--.

\bibitem{laptev2000recent}
A.~Laptev and T.~Weidl.
\newblock {R}ecent results on {L}ieb-{T}hirring inequalities.
\newblock {\em Journ{\'e}es {\'E}quations aux d{\'e}riv{\'e}es partielles},
  pages 1--14, 2000.

\bibitem{laptev2000gammathreehalves}
A.~Laptev and T.~Weidl.
\newblock {S}harp {L}ieb-{T}hirring inequalities in high dimensions.
\newblock {\em Acta Math.}, 184(1):87--111, 2000.

\bibitem{lehoucq1998arpack}
R.B. Lehoucq, D.C. Sorensen, and C.~Yang.
\newblock {\em {ARPACK} users' guide: solution of large-scale eigenvalue
  problems with implicitly restarted {A}rnoldi methods}, volume~6.
\newblock Siam, 1998.

\bibitem{levitt}
A.~Levitt.
\newblock {C}onvergence of gradient-based algorithms for the {H}artree-{F}ock
  equations.
\newblock {\em ESAIM Math. Model. Numer. Anal.}, 46(06):1321--1336, 2012.

\bibitem{liebclr}
E.H. Lieb.
\newblock {T}he {N}umber of {B}ound {S}tates of {O}ne-{B}ody {S}chr{ö}dinger
  {O}perators and the {W}eyl {P}roblem.
\newblock {\em Bull. Amer. Math. Soc.}, 82:751--753, 1976.

\bibitem{lieb1981thomas}
E.H. Lieb.
\newblock {T}homas-{F}ermi and related theories of atoms and molecules.
\newblock {\em Rev. Mod. Phys.}, 53(4):603, 1981.

\bibitem{lieb1990stability}
E.H. Lieb.
\newblock {T}he stability of matter: from atoms to stars.
\newblock {\em Bull. Amer. Math. Soc.}, 22(1), 1990.

\bibitem{liebseiringer}
E.H. Lieb and R.~Seiringer.
\newblock {\em {T}he {S}tability of {M}atter in {Q}uantum {M}echanics}.
\newblock Cambridge Univ. Press, 2010.

\bibitem{lieb1977hartree}
E.H. Lieb and B.~Simon.
\newblock {T}he {H}artree-{F}ock theory for {C}oulomb systems.
\newblock {\em Comm. Math. Phys.}, 53(3):185--194, 1977.

\bibitem{lieb1975bound}
E.H. Lieb and W.E. Thirring.
\newblock {B}ound for the {K}inetic {E}nergy of {F}ermions {W}hich {P}roves the
  {S}tability of {M}atter.
\newblock {\em Phys. Rev. Lett.}, 35:687--689, 1975.

\bibitem{lieb1976ineq}
E.H. Lieb and W.E. Thirring.
\newblock {I}nequalities for the moments of the eigenvalues of the
  {S}chrodinger {H}amiltonian and their relation to {S}obolev inequalities.
\newblock {\em Studies in Math. Phys., Essays in Honor of Valentine Bargmann},
  1976.

\bibitem{netrusov1996lieb}
Y.~Netrusov and T.~Weidl.
\newblock {O}n {L}ieb-{T}hirring inequalities for higher order operators with
  critical and subcritical powers.
\newblock {\em Comm. Math. Phys.}, 182(2):355--370, 1996.

\bibitem{roothaan1951}
C.C.J. Roothaan.
\newblock {N}ew developments in molecular orbital theory.
\newblock {\em Rev. Mod. Phys.}, 23(2):69--89, 1951.

\bibitem{rozenblum1972}
G.V. Rozenblum.
\newblock {D}istribution of the discrete spectrum of singular differential
  operators.
\newblock {\em Soviet Math. Doki.}, 202:1012--1015, 1972.

\bibitem{rumin2011balanced}
M.~Rumin.
\newblock {B}alanced distribution-energy inequalities and related entropy
  bounds.
\newblock {\em Duke Math. J.}, 160(3):567--597, 2011.

\bibitem{stein1971introduction}
E.M. Stein and G.L. Weiss.
\newblock {\em {I}ntroduction to {F}ourier analysis on {E}uclidean spaces}.
\newblock Princeton Univ. Press, 1971.

\bibitem{weidl1996}
T.~Weidl.
\newblock {O}n the {L}ieb-{T}hirring constants ${L}_{\gamma,1}$ for $\gamma
  \geq 1/2$.
\newblock {\em Comm. Math. Phys.}, 178(1):135--146, 1996.

\end{thebibliography}

\end{document}